%% file: process.tex
\documentclass [12pt]{article}
\usepackage {a4}
\usepackage {amsfonts}
\usepackage {amssymb,amsmath,amsthm}
\usepackage [utf8]{inputenc}
\usepackage [english]{babel}
\usepackage{enumerate}
\usepackage{graphicx}
\usepackage{fullpage}
\usepackage{bbm}
\usepackage{hyperref}

\input{env.tex}

\author{Marcin Kotowski \and B{\'a}lint Vir{\'a}g}
\title{Dyson's spike for random Schroedinger operators and Novikov-Shubin invariants of groups}

\begin{document}
\maketitle

\begin{abstract}
We study Schroedinger operators with random edge weights and their expected spectral measures $\mu_H$ near zero. We prove that the measure exhibits a spike of the form $\mu_H(-\varepsilon,\varepsilon) \sim \frac{C}{\abs{\log\varepsilon}^2}$ (first observed by Dyson), without assuming independence or any regularity of edge weights. We also identify the limiting local eigenvalue distribution, which is different from Poisson and the usual random matrix statistics. We then use the result to compute Novikov-Shubin invariants for various groups, including lamplighter groups and lattices in the Lie group Sol. 
\end{abstract}

\tableofcontents

\input{sec0.tex}

\input{sec2.tex}

\input{sec3.tex}

\input{sec3a.tex}
\input{sec3b.tex}

\input{sec4.tex}

\input{sec5.tex}

\subsection*{Acknowledgements}

The work was supported by the Canada Research Chair program, the Marie Curie grant SPECTRA, the OTKA grant $K109684$, and the NSERC Discovery Accelerator Program.

\bibliography{biblio}{}
\bibliographystyle{amsalpha}

\end{document}

%% file: env.tex
\newenvironment{psmallmatrix}
  {\left(\begin{smallmatrix}}
  {\end{smallmatrix}\right)}

\newcommand{\id}{\mathbbm{1}}
\newcommand{\scalar}[2]{\left\langle #1, #2\right\rangle}

\newcommand{\R}{\mathbb{R}}
\newcommand{\A}{\mathbb{A}}
\newcommand{\C}{\mathbb{C}}
\newcommand{\Pp}{\mathbb{P}}
\newcommand{\E}{\mathbb{E}}

\newcommand{\Z}{\mathbb{Z}}

\newcommand{\abs}[1]{\left\vert #1 \right\vert}

\newcommand{\Var}{\mathrm{Var}}
\newcommand{\rre}{\mathrm{Re}}
\newcommand{\schr}{\mathrm{Schr}}

\newtheorem{theorem}{Theorem}[section]
\newtheorem{definition}[theorem]{Definition}
\newtheorem{lemma}[theorem]{Lemma}

\newtheorem{proposition}[theorem]{Proposition}
\newtheorem{condition}[theorem]{Condition}

\newtheorem{remark}[theorem]{Remark}
\theoremstyle{definition}

%% file: sec0.tex
\section{Introduction}

In this paper, we study a class of Schroedinger operators with random edge weights given by:
\[
(Hf)(i) = a_{i-1}f(i-1) + a_{i}f(i+1)
\]
where $a_i$ form a stationary process. One can think of such an operator as a perturbation of the standard adjacency operator on $\Z$ (with all $a_i=1$) by random noise. It is well known that the presence of even small amount of noise can dramatically influence the spectral properties of $H$. While the spectral measure of the standard adjacency operator is absolutely continuous, the noisy variant typically has a fully discrete spectrum with exponentially localized eigenfunctions. This phenomenon is known as Anderson localization and is well studied in mathematical physics.

A natural object of study is the {\it expected spectral measure}, denoted $\mu_{H}$, which is the spectral measure of $H$ averaged over all random instances. Here, another phenomenon occurs, related to the behavior of $\mu_H$ at zero. The standard adjacency operator on $\Z$ has spectral measure $\mu$ with bounded density near zero, so in particular $\mu(-\varepsilon,\varepsilon) \sim \frac{\varepsilon}{\pi}$. In contrast, $\mu_H$ can exhibit behavior of the form:
\[
\mu_H(-\varepsilon,\varepsilon) \sim \frac{C}{\abs{\log\varepsilon}^2},
\] 
which goes to $0$ as $\varepsilon\rightarrow 0$ slower than any power of $\varepsilon$. If $\mu_H$ happens to have a density, this means it must have a sharp spike near $0$ of the form $ \frac{1}{\varepsilon}\cdot\frac{C}{\abs{\log\varepsilon}^3}$. This phenomenon was first observed in the famous paper by Dyson \cite{dyson}, who proved $\mu_H(-\varepsilon,\varepsilon) \sim \frac{C}{\abs{\log\varepsilon}^2}$ for a specific choice of edge weight distribution. One could expect that such behavior should be typical, independent of any particular properties of the distribution, and indeed, heuristic and numerical arguments supporting this claim have been given in the physics literature \cite{eggarter}. However, despite over 60 years from Dyson's original paper, the only rigorous general result in this direction was \cite{campanino}, where authors prove a lower bound $\mu_H(-\varepsilon,\varepsilon) \geq \frac{C}{\abs{\log\varepsilon}^3}$ for independent weights with bounded continuous density supported away from zero.

%
%\label{eq:intro-main}
%\begin{equation}
%\mu_T(-\varepsilon,\varepsilon) = \frac{C}{\abs{\log\varepsilon}^2}(1 + o(1))
%\end{equation}
%for general edge weight distributions that have finite moments of sufficiently high order. 

In this paper, we settle the question, proving:

\begin{theorem}\label{th:intro-main}
Let $a_i$ be i.i.d. random variables such that $\sigma^2 = \Var \log\abs{a_i} < \infty$. Then for the random Schroedinger operator $H$ defined by:
\[
(Hf)(i) = a_{i-1}f(i-1) + a_{i}f(i+1)
\]
we have
\begin{equation}\label{eq:intro-main}
\mu_H(-\varepsilon, \varepsilon) = \frac{\sigma^2}{\abs{\log^2 \varepsilon}}(1 + o_{\varepsilon}(1))
\end{equation}
\end{theorem}
Notably, we do not assume any regularity of the distribution: it can even take finitely many values. A version of this theorem holds also for edge weights which are not independent, as long as they satisfy suitable correlation decay (see Theorem \ref{th:main} for precise statement). Our theorem easily reproduces Dyson's result from \cite{dyson}, see the discussion after Theorem \ref{th:main}. 

The crucial ingredient in the proof is truncating the operator $H$ to a finite interval and finding a discrete process that counts its eigenvalues. We then proceed to identify a scaling limit of this process, which involves the Brownian Motion arising as the limit of the discrete random walk with steps $\log\vert \frac{a_{2i-1}}{a_{2i}} \vert$.

In the course of proving the theorem, we also establish eigenvalue bounds for finite Schroedinger operators. These are expressed in terms of up- and down-crossings, see Sections \ref{section:setup} and \ref{sec:conv-bounds} for details and precise definitions:

\begin{proposition}\label{prop:intro-finite}
Let $H_n$ be a Schroedinger operator of size $n$, with $n$ odd, with edge weights $a_i$. Let $M_n^{\lambda}$ denote the number of eigenvalues of $H_n$ inside the interval $(0,\lambda)$. Let $D^{1\pm\delta}_n$ denote the number of $(1\pm\delta) \abs{\log\lambda}$-downcrossings made by the process $\sum_{i=1}^{k}\log\abs{\frac{a_{2i-1}}{a_{2i}}}$. Let $B^{\delta}_n$ be denote the number of $k$ for which $\abs{\log\abs{a_{k}}} > \frac{\delta}{8}\abs{\log\lambda}$. Then for any $\delta \in (0,1) $ and $\lambda < \left( \frac{\delta}{16 n} \right)^{\frac{2}{\delta}}$ we have:
\[
D^{1+\delta}_n  -  2B^{\delta}_n  \leq M_n^{\lambda}\leq D^{1-\delta}_n  +  2B^{\delta}_n 
\]
\end{proposition}

Under the same assumptions, the analysis also yields an explicit limit of the local eigenvalue statistics around $0$. The limiting distribution has a simple description in terms of up- and down-crossings, or wells, made by the underlying Brownian Motion. Figure \ref{fig:wells-intro} shows an example of wells made by a Brownian Motion. We highlight that the limiting distribution is different from Poisson or the usual random matrix statistics.

%\begin{theorem}\label{th:intro-local}
%Let $a_i$ be i.i.d. random variables such that $\sigma^2 = \Var \log\abs{a_i} < \infty$. Let $H$ be the random Schroedinger operator defined by:
%\[
%(Hf)(i) = a_{i-1}f(i-1) + a_{i}f(i+1)
%\]
%and let $H_n$ the restricton of $H$ to an interval of size $n$. Then for any $t,s>0$ the number of eigenvalues of $H_{tn}$ inside the interval $(0,e^{-s\sqrt{n}})$ converges in distribution to $\left\lceil\frac{1}{2}J(t,s)\right\rceil$, where $J(t,s)$ is the number of subsequent $s$-crossings that a Brownian Motion with variance $\sigma^2$ makes inside the interval $[0,t]$ (see Section \ref{sec:conv-local} for definitions).
%\end{theorem}

\begin{figure}[h!]
  \centering
\includegraphics[scale=0.5]{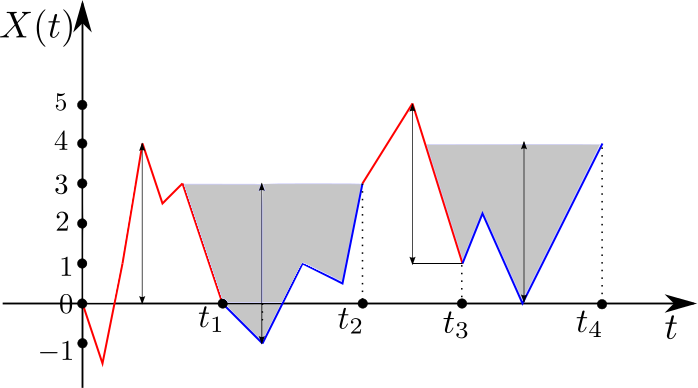}
  \caption{A random walk path with two wells of depth $4$}
\label{fig:wells-intro}
\end{figure}

\begin{theorem}\label{th:local}
Let $a_i$ be i.i.d. random variables such that $\sigma^2 = \Var \log\abs{a_i} < \infty$. Let $H$ be the random Schroedinger operator defined by:
\[
(Hf)(i) = a_{i-1}f(i-1) + a_{i}f(i+1)
\]
and let $H_n$ the restricton of $H$ to an interval of length $n+1$. Let $\{\Lambda_n(t,\eta), t,\eta > 0\}$ be the process equal to the number of eigenvalues of $H_{\lfloor tn \rfloor}$ inside the interval $(0,e^{-\eta\sqrt{n}})$. Let $\{\Lambda(t,\eta), t,\eta > 0\}$ be the process equal to the total number of disjoint $\eta$-wells that a Brownian Motion with variance $\sigma^2$ makes inside the interval $[0,t]$ (see Definition \ref{def:well}).  Consider a subsequence of $n$ such that $\lfloor tn \rfloor$ is odd.

Then the process $\Lambda_n(t,\eta)$ converges to the process $\Lambda(t,\eta)$, i.e. all finite dimensional distributions of $\Lambda_n$ converge weakly to finite dimensional distributions of $\Lambda(t,\eta)$.
\end{theorem}

As a corrollary, for the smallest positive eigenvalue $\lambda_0^{(n)}$ of $H_n$, we have:
\[
\frac{-\log\lambda^{(n)}_0}{\sqrt{n}} \ \Longrightarrow \ \sigma \cdot \sup_{t \in [0,1]}\abs{B(t)}
\]
in distribution, where $B(t)$ is a Brownian Motion of variance $1$ (see Remark \ref{rm:smallest-eigen}).	 

We then proceed to apply our result to computation of topological invariants of groups called the Novikov-Shubin invariants. Let $G$ be a finitely generated group and $H \in \C[G]$ a group ring element that determines a self-adjoint operator $H:\ell^2(G)\rightarrow \ell^2(G)$ with spectral measure $\mu_H$. The Novikov-Shubin invariant of $H$, denoted $\alpha(H)$, is determined by the behavior of $\mu_H$ at zero:
\[
\alpha(H) := \liminf_{\varepsilon\rightarrow 0}\frac{\log(\mu_H(-\varepsilon,\varepsilon) - \mu_H(\{0\}))}{\log\varepsilon}
\]
Informally, if $\alpha(H) = \alpha$ and $\alpha > 0$, this means that $\mu_H(-\varepsilon,\varepsilon)$ behaves like $\sim\varepsilon^{\alpha}$. In general, Novikov-Shubin invariants are rather difficult to compute, but carry interesting topological information about manifolds with fundamental group $G$ -- see \cite{eckmann}, \cite{luck} for further background on $\alpha(H)$.

Of particular interest is the question of positivity of $\alpha(H)$. Note, for example, that behavior of the form $\mu_H \sim \frac{C}{\abs{\log\varepsilon}^2}$ implies $\alpha(H)=0$. It was conjectured by Lott and L{\"u}ck \cite{lottluck} that $\alpha(H) > 0$ for any group $G$ and any $H \in \C[G]$. This has been recently disproved in \cite{balint-lukasz}. Our paper provides a new counterexample:

\begin{theorem}\label{th:intro-sol}
There exists a group $G=\Z^2 \rtimes_{A} \Z$, with $A$ a hyperbolic matrix, and $H \in \C[G]$ corresponding to a random walk on $G$ such that $\alpha(H)=0$.
\end{theorem}
In comparison to previous work, this gives a simple counterexample which is not only finitely presented, but also a lattice in a Lie group (Sol), see Section \ref{sec:hyperbolic} for a more precise statement. In conjunction with \cite{balint-lukasz}, our technique can be also used to prove $\alpha(H)=0$ for lamplighter groups $\Gamma \wr \Z$ with $\Gamma$ arbitrary, see Section \ref{sec:lamplighter}.

The connection between group theory and random Schroedinger operators comes in the form of a construction due to \cite{graboluk} that allows one, for certain semidirect products $G$, to build, given $H \in \C[G]$, a random Schroedinger operator whose expected spectral measure is equal to $\mu_H$ (see Section \ref{sec:hyperbolic} for details). To compute this measure, one needs the full power of our main theorem. For lamplighters, one obtains i.i.d. edge weights, but the underlying distribution can be discrete (e.g. for $\Z_2 \wr \Z$), so it is important that we have \eqref{eq:intro-main} without any smoothness assumption about the edge weights. In the case $G = \Z^2 \rtimes_{A} \Z$, the resulting operator has dependent edge weights, which we can handle thanks to the hyperbolic nature of the underlying map.

The paper is organized as follows. In Section \ref{section:convergence}, we restrict $H$ to a finite interval and derive a process that counts the eigenvalues of this restriction. This part is completely deterministic. In Section \ref{section:main}, we use that process to prove the main theorem which establishes equation \eqref{eq:intro-main} under suitable assumptions. Then, in Section \ref{sec:groups}, we describe the connection to group theory and proceed to apply the main theorem to computations for the groups mentioned above.

%% file: sec2.tex
\section{The eigenvalue process and its limit}\label{section:convergence}

\subsection{The expected spectral measure of $H$}\label{sec:spectral-measure}

Let $\Omega$ be a probability space and let $(\dots,a_{-1},a_0,a_{1},\dots):\Omega \rightarrow \R^{\Z}$ be a bi-infinite sequence of real numbers drawn from some joint shift-invariant probability distribution. We will always assume that almost surely, none of $a_i$ are equal to $0$. The distribution defines a random Schroedinger operator $H$ given by:
\[
(Hf)(i) = a_{i-1} f(i-1) + a_{i}f(i+1)
\]

In this section, we define what is meant by the expected spectral measure of $H$. There are standard definitions of this object, but we want to avoid technicalities coming from (i) the fact that the moments of the $a_i$ may not exist; and (ii) having to find the domain of $H$ on which it is self-adjoint. 

Let $\mu,\nu$ be probability measures. The Kolmogorov distance $d_{K}(\mu,\nu)$ is defined as $d_{K}(\mu,\nu)=\sup_{x}\abs{\mu(-\infty,x]-\nu(-\infty,x]}$. Note that if a sequence of measures $\mu_n$ forms a Cauchy sequence with respect to $d_{K}$, it converges weakly to some measure $\mu$.

Let $H_n$ denote the finite dimensional operator equal to the restriction of $H$ to the set $\{1,\dots,n+1\}$ by setting $a_i=0$ for $i\notin\{1,\dots,n\}$. In this way, we obtain a random finite dimensional operator $H_{n}$ whose (random) spectral measure $\mu_{H_{n}}$ is defined as its empirical eigenvalue distribution. Let $\mu_{n}$ denote the expected spectral measure of $H_{n}$, i.e. the average of $\mu_{H_{n}}$ take over the randomness of edge weights.

\begin{proposition}\label{prop:spectral-exists}
The sequence of measures $\{\mu_{2^{k}-1}\}_{k=1}^{\infty}$ converges weakly to some limit measure $\mu$ that satisfies $d_{K}(\mu,\mu_{2^{k}-1}) \leq \frac{1}{2^{k}}$.
\end{proposition}

\begin{proof}
For $k\geq 1$, consider the operators $H_{2^{k}-1}$ and $H_{2^{k+1}-1}$. Let $\widetilde{H}_{2^{k+1}-1}$ denote the operator obtained from $H_{2^{k+1}-1}$ by setting $a_{2^{k}}=0$. Since the matrix $H_{2^{k+1}-1} - \widetilde{H}_{2^{k+1}-1}$ has a possibly nonzero entry only at $a_{2^{k}}$, we have $\mathrm{rank}(H_{2^{k+1}-1} - \widetilde{H}_{2^{k+1}-1}) \leq 1$. By a standard matrix inequality \cite[Theorem A.43]{baisilv}:
\begin{equation}\label{eq:kolmogorov}
d_K(\mu_{H_{2^{k+1}-1}}, \mu_{\widetilde{H}_{2^{k+1}-1}}) \leq \frac{1}{2^{k+1}}\cdot\mathrm{rank}(H_{2^{k+1}-1} - \widetilde{H}_{2^{k+1}-1}) \leq \frac{1}{2^{k+1}}
\end{equation}
Note that the matrix $\widetilde{H}_{2^{k+1}-1}$ consists of two disjoint blocks corresponding to vertices $\{1,\dots,2^k\}$ and $\{2^k+1,\dots,2^{k+1}\}$. This and the shift-invariance of the distribution of $(\dots,a_{-1},a_0,a_{1},\dots)$ implies that the expected spectral measure of $\widetilde{H}_{2^{k+1}-1}$ is the same as the expected spectral measure of $H_{2^{k}-1}$. Thus, inequality \eqref{eq:kolmogorov} holds after taking expectations:
\begin{equation}\label{eq:kolmogorov-exp}
d_K(\mu_{2^k}, \mu_{2^{k+1}}) \leq \frac{1}{2^{k+1}}
\end{equation}
Altogether, \eqref{eq:kolmogorov-exp} implies that the sequence of measures $\mu_{2^{k}}$ is a Cauchy sequence, so it converges weakly to a measure $\mu$. The bounds used above easily imply that $d(\mu,\mu_{2^{k}}) \leq \sum_{i=k}^{\infty}\left(\frac{1}{2}\right)^{i+1} = \frac{1}{2^{k}}$ as claimed. 
\end{proof}

\begin{definition}\label{def:spectral-measure}
The expected spectral measure $\mu_T$ of the random Schroedinger operator $H$ is defined as the limit measure $\mu$ constructed in Proposition \ref{prop:spectral-exists}.
\end{definition}

Note that in the case where all $a_i$ are bounded, $H$ is a bounded operator on $\ell^2(\Z)$. One can then define the spectral measure of an instance $H_{\omega}$ as follows. Let $\delta_0(0)=1, \delta_0(i)=0$  for $i \neq 0$. The spectral measure $\mu_{H_{\omega}}$ is defined via specifying its moments:
\begin{equation}\label{eq:measure-unique}
m_k = \scalar{\delta_{0}}{(H_{\omega})^k \delta_0} = \int_{\R}x^k d\mu_{H_{\omega}}
\end{equation}
Since the operator $H_{\omega}$ is bounded, the moment sequence $m_k$ specifies the measure uniquely. One then defines the expected spectral measure $\mu_{H}$ simply as the expectation of the measures $\mu_{H_{\omega}}$. The approximation by finite operators claimed in Proposition \ref{prop:spectral-exists} is then obtained by a bounded operator analogue \cite[Lemma 6.1]{mean-quantum-percolation} of inequality \eqref{eq:kolmogorov}.

\subsection{Transfer matrices}\label{sec:transfer}

Throughout the following sections, we are concerned with a single instance $H_{\omega}$, which we will call $H$ from now on. Therefore, $H$ is a fixed, deterministic operator -- we introduce the probabilistic part of the analysis in Section \ref{section:main}.

To control the spectral measure of $H$, we will approximate $H$ by operators $H_n$ supported on finite intervals. In Section \ref{sec:transfer}, we use the standard transfer matrix approach to derive a process that counts the eigenvalues of $H_n$. After the setup contained in Section \ref{section:setup}, we analyze the process in Section \ref{sec:conv-bounds} and find its limiting behavior in Section \ref{sec:conv-limit}. 

Let $H_n$ be the restriction of $H$ to the set $\{1,\dots,n+1\}$, i.e. the operator obtained from $H$ by putting all weights outside the interval $[1,n+1]$ equal to $0$. For any given $\lambda \in \R$, we are interested in computing the number of eigenvalues of $H_n$ inside the interval $[0,\lambda]$. We shall perform this computation using the transfer matrix approach.

We start with the eigenvalue equation. For the first and last equation we set $a_{0} = a_{n+1} =1, \phi_{0} = \phi_{n+2}=0$. The eigenvalue equation can be then written as:
\begin{align*}
 H_n \phi &= \lambda \phi \\
 a_{k-1} \phi_{k-1} - \lambda \phi_{k} + a_{k}\phi_{k+1} &= 0, \ k = 1, 2, \dots, n+1 \\
% \lambda \phi_1 &= a_1 \phi_2\\
% \lambda \phi_{n+1} &= a_{n}\phi_{n}
%& \lambda \phi_{0} + a_{0}\phi_{1} = 0, \ n = 0 \\
%& a_{N-1} \phi_{N-1} - \lambda \phi_{N} = 0, \ n = N \\
\end{align*}
Letting:
\[
T_{k-1}^{\lambda}
=
\begin{pmatrix}
\frac{\lambda}{a_{k}} & -\frac{a_{k-1}}{a_{k}} \\
1 & 0 
\end{pmatrix}
,\
k = 1, 2, \dots, n+1
\]
we can write the recursion as: 
\begin{align*}
\begin{pmatrix}
\phi_{k+1} \\
\phi_{k} 
\end{pmatrix}
=
T_{k-1}^{\lambda}
\begin{pmatrix}
\phi_{k} \\
\phi_{k-1} 
\end{pmatrix}
\end{align*}
In particular, $\lambda$ is an eigenvalue if and only if for some $c$ we have:
\begin{equation}\label{eq:eigen}
\begin{pmatrix}
0 \\
c
\end{pmatrix}
=
T_{n}^{\lambda}\cdot \dots \cdot T_{0}^{\lambda}
\begin{pmatrix}
1 \\
0 
\end{pmatrix}
\end{equation}

We will be interested in the evolution of $(T_k^{0} \cdot \dots \cdot T_{0}^{0})^{-1} \cdot T_k^{\lambda} \cdot \dots \cdot T_{0}^{\lambda}$ as $k$ changes from $0$ to $n$. Let:
\begin{align*}
R_{k}^{\lambda} := (T_{k}^{0})^{-1} T_{k}^{\lambda} = 
\begin{pmatrix}
0 & 1 \\
-\frac{a_{k+1}}{a_{k}} & 0 
\end{pmatrix}
\begin{pmatrix}
\frac{\lambda}{a_{k+1}} & -\frac{a_{k}}{a_{k+1}} \\
1 & 0 
\end{pmatrix}
= 
\begin{pmatrix}
1 & 0 \\
-\frac{\lambda}{a_{k}} & 1 
\end{pmatrix}
\end{align*}
Now we rewrite:
\begin{align}\label{eq:process}
(T_k^{0} \cdot \dots \cdot T_{0}^{0})^{-1} \cdot T_k^{\lambda} \cdot \dots \cdot T_{0}^{\lambda} = 
(R_{k}^{\lambda})^{T_{k-1}^{0} \cdot \dots \cdot T_{0}^{0}} \cdot  (R_{k-1}^{\lambda})^{T_{k-2}^{0} \cdot \dots \cdot T_{0}^{0}} \cdot \dots \cdot R_{0}^{\lambda}
\end{align}
where $R_{k}^{A} = A^{-1} R_{k} A$.

It is desirable to express \eqref{eq:process} in a more tractable way. Define:
\begin{align*}
%&U_i =  \\
S_k = \sum_{i=1}^{k} 2\log\left\vert\frac{a_{2i-1}}{a_{2i}}\right\vert 
\end{align*}

We first compute products of odd and even numbers of $T_{k}^{0}$. We have:
\begin{align}
T_{k+1}^{0}T_{k}^{0} = 
\begin{pmatrix}
0 & -\frac{a_{k+1}}{a_{k+2}} \\
1 & 0 
\end{pmatrix}
\begin{pmatrix}
0 & -\frac{a_{k}}{a_{k+1}} \\
1 & 0 
\end{pmatrix}
=
\begin{pmatrix}
 -\frac{a_{k+1}}{a_{k+2}} & 0 \\
0 & -\frac{a_{k}}{a_{k+1}}   
\end{pmatrix}
\end{align}
so:
\begin{align}
T_{2k}^{0} \cdot T_{2k-1}^{0} \cdot \dots \cdot T_{1}^{0} \cdot T_{0}^{0} &= 
(-1)^k
\begin{pmatrix}
%0 & -\frac{a_{2k}\cdot a_{2k-2}\cdot\dots \cdot a_{2}}{a_{2k+1} \cdot a_{2k-1} \cdot\dots \cdot a_{1} } \\
%\frac{a_{2k-1} \cdot a_{2k-3} \cdot\dots \cdot a_{1}}{a_{2k} \cdot a_{2k-2} \cdot\dots \cdot a_{2}} & 0
0 &   - \frac{\varepsilon_k \cdot e^{-\frac{1}{2}S_{k}}}{a_{2k+1}} \\
\varepsilon_k \cdot e^{\frac{1}{2}S_k}   & 0
\end{pmatrix}
\\
T_{2k+1}^{0}T_{2k}^{0} \cdot T_{2k-1}^{0} \cdot \dots \cdot T_{1}^{0} \cdot T_{0}^{0} &= 
(-1)^{k+1}\begin{pmatrix}
%\frac{a_{2k+1}\cdot a_{2k-1} \cdot \dots \cdot a_{3}}{a_{2k+2}\cdot a_{2k} \cdot \dots \cdot a_{2}} & 0 \\
%0  & \frac{a_{2k}\cdot a_{2k-2} \cdot \dots \cdot a_{2}}{a_{2k+1} \cdot a_{2k-1} \cdot\dots \cdot a_{3} } 
\varepsilon_{k+1} \cdot e^{\frac{1}{2}S_{k+1}} & 0 \\
0  & \frac{\varepsilon_k \cdot e^{-\frac{1}{2}S_k}}{a_{2k+1}}    
\end{pmatrix}
\end{align}
where we have written $\prod_{i=1}^{k}\frac{a_{2i-1}}{a_{2i}}$ as $\varepsilon_k \cdot e^{\frac{1}{2}S_k}$ with $\varepsilon_k = \pm 1$.

We now have:
\begin{align}\label{eq:rotation}
Q_{k}^{\lambda} := (R_{2k+1}^{\lambda})^{T_{2k}^{0} \cdot \dots \cdot T_{0}^{0}}(R_{2k}^{\lambda})^{T_{2k-1}^{0} \cdot \dots \cdot T_{0}^{0}}
& =
\begin{pmatrix}
1-  \frac{\lambda^2}{a^2_{2k+1}}   &  \lambda \frac{e^{-S_k}}{a^2_{2k+1}} \\
-\lambda e^{S_{k}}  & 1 
\end{pmatrix}
\end{align}

We will now prove a lemma that justifies the usefulness of the representation \eqref{eq:process}. Let $\mathbb{H}=\{z \in \C : \Im z \geq 0\}\cup\{\infty\}$ be the upper half plane. We can identify a vector $\begin{psmallmatrix}
 a\\
 b
\end{psmallmatrix}
$ with a point $z \in \mathbb{H}$ by letting $z = \frac{a}{b}$. In this identification, the vector $\begin{psmallmatrix}
 1\\
 0
\end{psmallmatrix}$ is mapped to $\infty$. We can translate the action of matrices on vectors into action on $\mathbb{H}$. Recall that matrices $A \in SL(2,\R)$ act by isometries of the hyperbolic plane $\mathbb{H}$ in the upper half plane model. A matrix:
\begin{align*}
A=
\begin{pmatrix}
 a & b \\
 c & d
\end{pmatrix}
\end{align*}
corresponds to the map $\mathcal{T}_A: \mathbb{H} \rightarrow \mathbb{H}$ such that $\mathcal{T}_A(z)= \frac{az+b}{cz+d}$. To simplify notation, we will write $A$ instead of $\mathcal{T}_A$.

From now on we assume that $n$ is odd. This implies that
\begin{equation}\label{eq:v}
(T_{n}^{0} \cdot \dots \cdot T_{0}^{0})^{-1}(0) = 0
\end{equation}

%We pass from the half plane $\mathbb{H}$ to the disc $\mathbb{D}$ by applying the map $\phi(z) = \frac{z-i}{z+i}$, which maps the real line to the boundary circle $S^1$. Thus $v=0$ will be mapped to $-1 \in S^1$ and $\infty$ will be mapped to $1 \in S^1$. Note that by \eqref{eq:eigen}, $(R_{n}^{\lambda})^{T_{n-1}^{0} \cdot \dots \cdot T_{0}^{0}} \cdot \dots \cdot R_{0}^{\lambda}(1) = -1$ is equivalent to $\lambda$ being an eigenvalue of $H_n$. 

%We equip the half plane boundary $\partial\mathbb{H}$ with the metric obtained from $S^1$ via the standard map $\phi: \partial\mathbb{H}\rightarrow S^1, \phi(z) = \frac{z-i}{z+i}$. In this way we obtain the notion of linear interpolation between points in $\partial\mathbb{H}$ in a natural way.

%\begin{definition}
%Let $z_0 = 1$ and $z_k = (R_{k-1}^{\lambda})^{T_{k-2}^{0} \cdot \dots \cdot T_{0}^{0}} (z_{k-1})$ for $k=1,\dots,n+1$. We connect each $z_{k}$ to $z_{k-1}$ along a positively oriented arc, obtaining a continuous path $\alpha: [0,1]\rightarrow S^1$. We define the number of half rotations made by the process $z_k$ as the number of times $\alpha(t) = 1$ or $-1$ for $t\neq0$.
%\end{definition}

\begin{lemma}\label{lm:topo}
Let $M^{\lambda}_{n}$ denote the number of eigenvalues of $H_n$ inside the interval $[0,\lambda]$ and let $J^{\lambda}_{n}$ denote the number of times the process $(R_{k}^{\lambda})^{T_{k-1}^{0} \cdot \dots \cdot T_{0}^{0}} \cdot \dots \cdot R_{0}^{\lambda}(\infty)$ passes $0$ or $\infty$ as $k$ ranges from $0$ to $n$. Then $M_{n}^{\lambda} = \lceil \frac{1}{2} J_{n}^{\lambda}\rceil $.
\end{lemma}
\begin{proof}
Let $A = [0,\lambda] \times [0, n]$. We define $f: [0,\lambda] \times \{0,\dots,n\} \rightarrow \partial\mathbb{H}$ by:
\[
f(\lambda^{\ast}, k) =  (R_{k}^{\lambda^{\ast}})^{T_{k-1}^{0} \cdot \dots \cdot T_{0}^{0}} \cdot \dots \cdot R_{0}^{\lambda^{\ast}}(\infty)
\]
and interpolate $f$ linearly to obtain a continuous map $f: A \rightarrow S^1$. Note that by \eqref{eq:eigen} and \eqref{eq:v}, $\lambda^{\ast}$ is an eigenvalue whenever:
\[
(R_{k}^{\lambda^{\ast}})^{T_{k-1}^{0} \cdot \dots \cdot T_{0}^{0}} \cdot \dots \cdot R_{0}^{\lambda^{\ast}}(\infty) = 0
\]

\begin{figure}[h!]
  \centering
\includegraphics[scale=0.25]{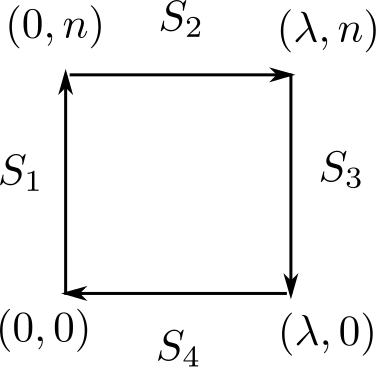}
  \caption{The loop from Lemma \ref{lm:topo}}
\label{fig:loop}
\end{figure}

Let $S_1, S_2, S_3, S_4$ be the sides of the rectangle $A$ (see Figure \ref{fig:loop}). Since $A$ is contractible and $f$ is continuous, the loop $f: S_1 \cup S_2 \cup S_3 \cup S_4 \rightarrow \partial\mathbb{H}$ has winding number zero. Since $f(S_1)=f(S_4)=\{\infty\}$, the same is true for the loop $f: S_2 \cup S_3 \rightarrow \partial\mathbb{H}$. Since $f$ is monotone on both of these intervals, with opposite direction, the number of times it passes $0$ on them is the same, which easily finishes the proof.
\end{proof}

%% file: sec3.tex
\subsection{Crossings and the rotation process}\label{section:setup}

In Section \ref{sec:transfer} we have shown that in order to study the number of eigenvalues of $H_n$, we have to study the time evolution of a process given by composing rotations (see equation \ref{eq:rotation} and remark thereafter). In the following sections, we study this process in detail and derive its continuous-time scaling limit. 

We consider a fixed, deterministic operator $H_n$ on the interval of size $n$, given by edge weights $\{a_{k}\}_{k=1}^{n}$. The process constructed in the previous section can be described informally as follows. Starting from the initial point $v_0 \in \R$, each point $v_k$ will be moved by a rotation $R_k$ to a new point $v_{k+1}$. The center of each rotation, equal to $i \frac{e^{-S_{k}}}{\abs{a_{2k+1}}}$, is obtained as a product of $a_i$ and $\lambda$ represents the speed of rotation. 

Formally, we study the process defined by:
\begin{align*}
& Q^{\lambda}_k = 
\begin{pmatrix}
1-  \frac{\lambda^2}{a_{2k+1}^2}  &  \lambda \frac{e^{-S_k}}{{a_{2k+1}^2}} \\
-\lambda e^{S_{k}} & 1 
\end{pmatrix}, \\
& v^{\lambda}_k = Q^{\lambda}_{k-1}(v^{\lambda}_{k-1}) \\
& v^{\lambda}_0 = \infty
\end{align*}
The process described in \eqref{eq:rotation} is exactly of this form. 

Let $\mathbb{H} = \{z \in \mathbb{C} : \mathrm{Im} z \geq 0\} \cup \{\infty\}$ and $\partial\mathbb{H}=\R\cup\{\infty\}$. Since we would like to study the points $\log v_k$ for $v_k\in\partial\mathbb{H}$, it is natural to introduce the following setup. For $k \in \Z$, let $\A_k = \{\R + k\cdot i\pi\} \cup \{-\infty,+\infty\}$. Let $\A = \bigsqcup_{k=1}^{\infty} \A_k$ be the union of lines plus points $\pm \infty$, connected in such a way that $\A_{2k} \cap \A_{2k+1} = \{+\infty\}, \A_{2k-1} \cap \A_{2k} = \{-\infty\}$. Considering $\exp: \A \rightarrow \partial\mathbb{H}$, we can treat its inverse as a (multi-valued) map $\log: \partial\mathbb{H} \rightarrow \A$, where $\A_{2k} \subseteq \log(\{z \geq 0\})$ and  $\A_{2k+1} \subseteq \log(\{z \leq 0\})$. 

\begin{remark}\label{rm:ordering}
Note that $\A$ has a natural ordering inherited from the real line, which we will denote by $\geq_{\A}$. If $x \in \A_{j}$ and $y \in \A_{i}$ for $j > i$, then $x \geq_{A}y$. If $x,y$ are in the same component, $x \geq_{\A} y$ means $x \geq y$ if $x,y \in \A_{2i}$ or $x \leq y$ if $x,y \in \A_{2i+1}$. The time evolution governed by $Q_k$ is monotone with respect to this ordering, i.e. if $y \geq_{\A} y'$, then $Q_k(y) \geq_{\A} Q_k(y')$.
\end{remark}

The processes studied below will consist of points starting at $+\infty \in \A_{1}$ and decreasing monotoneously until they jump past $-\infty$ to $\A_{2}$, whereupon they increase until they jump to $\A_{3}$, and so on. 

%We endow $\A$ with a metric in the following way. Let $\phi: \partial\mathbb{H}\rightarrow S^1$ be given by $\phi(z)=\frac{z-i}{z+i}$ and $\log: S^1 \rightarrow \R$ be the (multi-valued) logarithm. This naturally induces a map $\psi: \A \rightarrow \R$ given by $\psi = \log\circ\phi\circ\exp$, which can be taken to be single-valued by requiring e.g. $\psi(\A_k)=[(k-1)\pi,k\pi]$. We let $d_\A$ to be a metric on $\A$ obtained by pulling back the standard metric $d$ on $\R$ via $\psi$. The exact formula for the metric is not essential, but it is important that under $d_{\A}$ all distances in $\A$ are finite. For a fixed interval $[0,T]$, the processes considered in this sections will be functions $f: [0,T]\rightarrow \A$. 

%\begin{figure}[h!]
%  \centering
%\includegraphics[scale=0.35]{img/diagram.png}
%  \caption{The map $\psi$ induces a metric on $\mathbb{A}$.}
%\end{figure}

%We use the superscript $n$ to emphasize that the sequences for different values of $n$ are not related to one another. While it may seem natural to consider sequences obtained from a single infinite operator $H$ restricted to finite intervals, we will need the main theorem of this section stated in generality.

%We study the process at times $k=1,2,\dots,Tn$ for some $T>0$ and eventually we will let $n\rightarrow\infty$. For clarity we present proofs with $T=1$. We will write $z_k$ for $z^{\lambda}_k$ when $\lambda$ is implicit from the context. 

%\begin{figure}[h!]
%  \centering
%\includegraphics[scale=0.45]{img/log.png}
%  \caption{Movement of the process $v_k$ after taking $\log$ and rescaling.}
%\end{figure}

We introduce the scaled version of the processes, defined as follows:
\begin{align*}
X_{k} &= \frac{1}{\abs{\log\lambda}}S_{k}\\
Y_{k} &= \frac{1}{\abs{\log\lambda}} \log v^{\lambda}_{k}
\end{align*}
$X_{k}$ is supposed to represent the rescaled motion of the rotation center, while $Y_{k}$ is the rescaled trajectory of $v_k$. With this setup, $Y_{k}$ takes values in $\A$, either in $\A_{2i}$ if $v^{\lambda}_{k} \geq 0$ or $\A_{2i+1}$ if $v^{\lambda}_{k} \leq 0$. It jumps from $\A_{i}$ to $\A_{i+1}$ whenever $v^{\lambda}_{k}$ changes sign. 

Note that for $v^{\lambda}_{k-1} > 0$, so that $Y_{k-1} \in \A_{2i}$ for some $i$, we can write:
\[
Y_{k} = \frac{1}{\abs{\log\lambda}}\log \abs{\frac{(1 - \frac{\lambda^2}{a_{2k-1}^2})\lambda^{-Y_{k-1}} + \frac{1}{a_{2k-1}^2}\lambda^{1+X_{k-1}}}{1 - \lambda^{1 -Y_{k-1}-X_{k-1}}}}
\]
where $Y_k \in \A_{2i}$ if the expression under the absolute value is nonnegative and $Y_k \in \A_{2i+1}$ otherwise.

\begin{remark}\label{rm:y-jump}
Assume that $v_{k-1} > 0$ and $1 - \frac{\lambda^2}{a^{2}_{2k-1}} > 0$. Then $v_k$ makes a jump, i.e. $v_k \leq 0$, if and only if $1 - X_{k-1} - Y_{k-1} \leq 0$.
\end{remark}

We now introduce the concept of crossings, which relate to the number of jumps made by the process. Let $X(t)$ be a real valued function from $\R$ or $\Z$. Define
\begin{align*}
& M(t_1, t_2) = \max_{t \in [t_1, t_2]}X(t)\\
& m(t_1, t_2) = \min_{t \in [t_1, t_2]}X(t)
\end{align*}

\begin{definition}\label{def:crossing}
Let $\tau_0 = 0$ and define for $i \geq 1$:
\begin{align*}
\tau_{2i-1} &= \inf\{s \geq  \tau_{2i-2} \ \vert \ M(\tau_{2i-2}, s) - X(s) \geq \alpha \} \\
\tau_{2i} &= \inf\{s \geq  \tau_{2i-1} \ \vert \ m(\tau_{2i-1}, s) - X(s) \leq -\alpha \} 
\end{align*}
We will say that $X$ has made a $\alpha$-downcrossing at time $\tau_{2i-1}$ and a $\alpha$-upcrossing at time $\tau_{2i}$, for $i\geq 1$. Both down- and upcrossings will be called crossings. 
\end{definition}

\begin{definition}\label{def:alpha-crossing-process}
Given a function $X(t)$ and $\alpha > 0$, the $\alpha$-crossing process associated to $X$ is defined as:
\[
Z^{\alpha}(t) =
\begin{cases}
-M(\tau_{2i},t)+\frac{\alpha}{2} \in \A_{2i+1},& \text{for }t \in [\tau_{2i},\tau_{2i+1}) \\
-m(\tau_{2i-1},t)-\frac{\alpha}{2} \in \A_{2i}, & \text{for }t \in [\tau_{2i-1},\tau_{2i})
\end{cases}
\]
where $\tau_{k}$ are times of subsequent crossings as in Definition \ref{def:crossing}.
\end{definition}

Informally, the proces $Z^{\alpha}(t)$ evolves as follows. It is started in $\A_{1}$ at $\frac{\alpha}{2}$ and decreases monotoneously as it is "pushed" by $-M(\tau_{0},t)+\frac{\alpha}{2}$. The moment $M(\tau_{0},t) - X(t) \geq \alpha$, $Z^{\alpha}(t)$ jumps to $\A_{2}$ and everything starts afresh, only with $-M(\tau_{0},t)+\frac{\alpha}{2}$ replaced with $-m(\tau_{1},t)-\frac{\alpha}{2}$ and moving in the opposite direction. The process jumps to $\A_{3}$ when $X(t) - m(\tau_{1},t) \geq \alpha$ and so on. 

\begin{definition}\label{def:jumps-y}
We say that a process $Y$ has made a jump at time $k$ if $Y_{k} \in \A_{j}$ and $Y_{k-1} \in \A_{j-1}$. 
\end{definition}

With these definitions, the first crossing is always a downcrossing. Note that the times of successive jumps of the $\alpha$-crossing process $Z^{\alpha}(t)$ are exactly the times of $\alpha$-crossings of the underlying function $X$.

%We will sometimes use $M(s,t), m(s,t)$ for $s=\tau_{i}$ and simply write $M(t), m(t)$, with the lower endpoint $s$ implicit from the context.  We will write $M_{k,l}, m_{k,l}, Z^{\alpha}_k$ instead of $M(s,t), m(s,t), Z^{\alpha}(t)$ when considering discrete time processes. 
%\begin{definition}\label{def:jumps-yn}
%We say that the process $Y^{\lambda}$ has made a jump at time $k$ if $Y^{\lambda}_{k} \in \A_{i}$ and $Y^{\lambda}_{k+1} \in \A_{i+1}$. 
%\end{definition}

%% file: sec3a.tex
\subsection{Upper and lower bounds on the number of jumps}\label{sec:conv-bounds}

%We expect that as as $n \rightarrow \infty$, the processes $Y$, with suitable dependence of $\lambda$ on $n$, will typically converge to the $2$-crossing process $Z^2$ with some underlying function $F$. Since the jumps of $Z^2$ are exactly the crossings of $F$, we expect a similar relation between the jumps of the finite process $Y$ and the crossings of $X$, or, equivalently, $S_k$.

In this Section, we prove upper and lower bounds on the number of jumps of $Y$ in terms of crossings made by $X$. The Propositions below will be used for $\lambda = e^{-\sqrt{n}}$ but are stated and proved in generality, with $\lambda$ arbitrary. Proposition \ref{prop:intro-finite} follows by recalling that by Lemma \ref{lm:topo}, the number of eigenvalues is $\lceil \frac{1}{2}J_n \rceil$ and that is simply the number of downcrossings.

%As $n\rightarrow\infty$, we expect the number of jumps made by $Y_n$ to be approximately equal to the number of $2$-crossings made by $X_n$. The following Proposition gives a more crude bound, bounding jumps of $Y_n$ by $(2-\delta)$-crossings of $X_n$. It has the advantage of being valid not only asymptotically, but also for finite $n$, which we shall exploit in Section \ref{section:main}

In both Propositions below, we consider the processes $Y$ and $Z^{\alpha}$, the $\alpha$-crossing process associated to $X$ for appropriate value of $\alpha$. We will only consider the case when $Y,Z \in \A_{2i}$, as the case of the odd numbered component is handled in a similar way. We can write the one step recursion for $Z$ as $Z_{k+1} = z'(X_{k+1}, Z_{k})$, where $z'$ is given by:
\[
z'(x,z)=
\begin{cases}
-x+\frac{\alpha}{2} \in \A_{2i+1},& \text{if }z \geq -x+\frac{\alpha}{2} \\
\max\{z, -x-\frac{\alpha}{2}\} \in \A_{2i}, & \text{otherwise}
\end{cases}
\]

Likewise, we can write the one step recursion for $Y$ as $Y_{k+1} = y'(X_k, Y_k, a_{2k+1})$, where $y'$ is given by:
\[
y'(x,y,a) = \frac{1}{\abs{\log\lambda}}\log \frac{(1 - \frac{\lambda^2}{a^2})\lambda^{-y} + \frac{1}{a^2}\lambda^{1+x}}{1 - \lambda^{1 -y-x}}
\]

In both proofs, we will repeatedly use the following estimates, which hold if $\abs{\log\abs{a_{k}}} \leq \frac{\delta}{16} \abs{\log\lambda}$, $\delta < 2$ and $\lambda < 1$:
\begin{align}
\label{eq:lambda-a2-1}
&\lambda^{\frac{\delta}{8}} < \frac{1}{a_{k}^{2}} < \lambda^{-\frac{\delta}{8}} \\ 
\label{eq:lambda-a2-2}
&1 - \frac{\lambda^{2}}{a_{k}^2} > 1 - \lambda^{2 - \frac{\delta}{8}} > 0
\end{align}

\begin{proposition}\label{prop:crossings}
Pick $\delta \in (0,2)$. Let $J_n$ be equal to the number of jumps made by $\{Y_{k}\}_{k=1}^{n}$. Let $C_n$ be equal to the number of $(2-\delta)\abs{\log\lambda}$-up- or down-crossings made by $S_k$ and let $B_n$ be the number of $k$ such that $\abs{\log\abs{a_{k}}} > \frac{\delta}{16}\abs{\log\lambda}$. Then for all $\lambda < \left( \frac{\delta}{32 n} \right)^{\frac{4}{\delta}}$ we have $J_n \leq 2 B_n + C_n$.
\end{proposition}

\begin{proof}
First, we replace $Y$ by a process defined as follows. Suppose that $Y_{k} \in \A_i$. Whenever $\abs{\log\abs{a_{k}}} > \frac{\delta}{16} \abs{\log\lambda}$, instead of following its usual evolution, the modified process jumps to $Y_{k} \in \A_{i+2}$, i.e. the same point, but two components ahead. Since $Y$ can make at most one jump in one step, the modified process is always ahead of $Y$, in particular, it makes at least as many jumps. Thus, it suffices to bound the number of jumps of $Y$ by $C_n$, under the assumption that $\abs{\log\abs{a_{k}}} \leq \frac{\delta}{16}\abs{\log\lambda}$, since the remaining jumps are taken care of by the term $2 B_n$. 

Let $Z$ be the $(2-\delta)$-crossing process associated to $X$ (Definition \ref{def:alpha-crossing-process}). Recall that $C_n$, the number of $(2-\delta)$-crossings made by $X$, is equal to the number of jumps made by $Z$. Thus we need to prove that $Z$ makes at least as many jumps as $Y$.

Recall the one step recursion for $Y$ and $Z$. Suppose that $Z_k\in \A_{i}$ and let $W_k = Z_k + (-1)^{i}\frac{\delta}{4n}\cdot k$. The evolution of $W_k$ is governed by the recursion:
\[
w'(x,w,k) = z'\left(x,w - (-1)^{i}\frac{\delta}{4n}\cdot k\right) + (-1)^{i}\frac{\delta}{4n}\cdot (k+1)
\] 
Note that $w'$ makes jumps at the same time as $z'$. We claim that in order to prove that $Z$ makes at least as many jumps as $Y$, it suffices to prove that for any values of $x, w,a,k$ such that $k \leq n$ and $a$ satisfies \eqref{eq:lambda-a2-1}, \eqref{eq:lambda-a2-2} we have:
\begin{equation}\label{eq:lower-bound-w'}
w'(x,w,k) \geq_{\A} y'(x,w,a)
\end{equation}
We prove inductively that $W_k \geq_{\A} Y_{k+1}$, from which the claim about the number of jumps follows as $W_k$ makes a jump if and only if $Z_k$ makes a jump. For the base case $k=0$, we have:
\begin{equation}\label{eq:base-case}
y'(\infty) = \frac{1}{\abs{\log\lambda}}\log\frac{1-\frac{\lambda^2}{a^2}}{\lambda} > 1 + \frac{1}{\abs{\log\lambda}}\log(1-\lambda^{2-\frac{\delta}{8}} ) > 1 - \frac{\delta}{8}
\end{equation}
where the last estimate follows from $\lambda^{\frac{\delta}{8}} < \frac{1}{2}$, easily implied by $\lambda < \left( \frac{\delta}{32 n} \right)^{\frac{4}{\delta}}$. Thus, $Y_1 > 1 - \frac{\delta}{8} > 1 - \frac{\delta}{2} = Z_0$. The map $y'$ is monotone, meaning that $y_1 \leq_{\A} y_2$ implies $y'(x,y_1,a) \leq_{\A} y'(x,y_2,a)$. The inductive step then follows by applying \eqref{eq:lower-bound-w'} and then monotonicity::
\[
W_{k+1} = w'(X_{k+1},W_k,k) \geq_{\A} y'(X_{k+1},W_k,a_{2k+3}) \geq_{\A} y'(X_{k+1},Y_{k+1},a_{2k+3}) = Y_{k+2}
\]
Note that the increments of $w'$ and $y'$ are translation invariant, so without loss of generality we can put $x=0$. The results hold for all values of $a_{k}$ satisfying the assumptions, so we suppress the variable $a$ and write $y(w)$ for $y(0,w,a)$.

We only consider the case when $y, z \in \A_{2i}$ as the case of the odd numbered component is handled in a similar way. We first claim that if $1 - \frac{\delta}{4}\geq u \geq -1 + \frac{\delta}{2}$, then $y'(u) - u < \frac{\delta}{4n}$. We have:
\begin{align*}
y'(u) - u &= \frac{1}{\abs{\log\lambda}}\log\frac{(1 - \frac{\lambda^2}{a^2})\lambda^{-u} + \frac{1}{a^2}\lambda}{1 - \lambda^{1-u}} - u \stackrel{  \eqref{eq:lambda-a2-1}}{\leq} 
\frac{1}{\abs{\log\lambda}}\log\frac{1 + \lambda^{1+u - \frac{\delta}{8}}}{1-\lambda^{1-u}} \leq \\
&\frac{1}{\abs{\log\lambda}}\log\frac{1 + \lambda^{\frac{3}{8}\delta}}{1-\lambda^{\frac{\delta}{4}}} \leq 
\frac{1}{\abs{\log\lambda}}\cdot \frac{\lambda^{\frac{\delta}{4}} + \lambda^{\frac{3}{8}\delta}}{1-\lambda^{\frac{\delta}{4}}} \leq 
\frac{1}{\abs{\log\lambda}} \cdot\frac{2\cdot\lambda^{\frac{\delta}{4}}}{1-\lambda^{\frac{\delta}{4}}}
\end{align*}
The assumptions $\lambda < \left (\frac{\delta}{32 n} \right)^{\frac{4}{\delta}}$ and $\delta < 2$ easily imply $\frac{1}{\abs{\log\lambda}} < 2$ and $\frac{1}{1-\lambda^{\frac{\delta}{4}}} < 2$, so we get:
\[
y'(u) - u < 8 \cdot \lambda^{\frac{\delta}{4}} < \frac{\delta}{4n}
\]
again by $\lambda < \left( \frac{\delta}{32 n} \right)^{\frac{4}{\delta}}$.

We now prove \eqref{eq:lower-bound-w'}. The first case to consider is when $y'$ does not jump, which by Remark \ref{rm:y-jump} and \eqref{eq:lambda-a2-2} implies $w < 1$. Then, either $w'$ jumps and the claim is trivial, or it does not. In that case, we have $w < 1 - \frac{\delta}{2} + \frac{\delta}{4n}\cdot k < 1 - \frac{\delta}{4}$ and $z' = \max\{w-\frac{\delta}{4n}\cdot k, -1+\frac{\delta}{2}\}$. If $w \geq -1+\frac{\delta}{2}+\frac{\delta}{4n}\cdot k$, then $z'=w-\frac{\delta}{4n}\cdot k$, so $w'=w + \frac{\delta}{4n}$. Also $1-\frac{\delta}{4}\geq w \geq -1 + \frac{\delta}{2}$, so $y'(w) -w < \frac{\delta}{4n} =w' -w$ as desired. If $w < -1+\frac{\delta}{2}+\frac{\delta}{4n}\cdot k$, then:
\[
y'(w) < y'\left(-1+\frac{\delta}{2} + \frac{\delta}{4n}\cdot k\right) < -1+\frac{\delta}{2} + \frac{\delta}{4n}\cdot k + \frac{\delta}{4n} = 
z' + \frac{\delta}{4n}\cdot (k+1) = w'
\] 

The other remaining case is when $y'$ makes a jump, so $w \geq 1$. This means that $z \geq 1 - \frac{\delta}{4n}\cdot k > 1 - \frac{\delta}{4} > 1 - \frac{\delta}{2}$, so $z'$ makes a jump as well. Also, $z' = 1 - \frac{\delta}{2} \in \A_{2i+1}$ and $w' > z'$. It thus suffices to prove that $y' \leq_{\A} 1 - \frac{\delta}{2}$ in $\A_{2i+1}$, which translates to $y' \geq 1 - \frac{\delta}{2}$ since we are in an odd numbered component. We have $y'(w) > y'(\infty)$ and $y'(\infty) > 1-\frac{\delta}{2}$ by \eqref{eq:base-case}.
\end{proof}

\begin{proposition}\label{prop:discrete-lower-bound}
Pick $\delta \in (0,2)$. Let $J_n$ be equal to the number of jumps made by $\{Y_{k}\}_{k=1}^{n}$. Let $C_n$ be equal to the number of $(2+\delta)\abs{\log\lambda}$-up- or down-crossings made by $S_k$ and let $B_n$ be the number of $k$ such that $\abs{\log\abs{a_{k}}} > \frac{\delta}{16} \abs{\log\lambda}$. Then for all $\lambda < 1$ we have $C_n - 2 B_n \leq J_n $.
\end{proposition}

\begin{proof}
As in Proposition \ref{prop:crossings}, we replace $Y$ by a process which makes two jumps whenever $\abs{\log\abs{a_{k}}} > \frac{\delta}{16} \abs{\log\lambda}$. Now it suffices to bound the number of jumps of $Y$ from below by $C_n$, under the assumption that $\abs{\log\abs{a_{k}}} \leq \frac{\delta}{16} \abs{\log\lambda}$, since the remaining jumps are taken care of by the term $2 B_n$.  

Let $Z$ be the $(2+\delta)$-crossing process associated to $X$ (Definition \ref{def:alpha-crossing-process}). We only consider the case when $Y_k, Z_k \in \A_{2i}$ as the other one is handled in a similar way. Recall that $C_n$, the number of $(2+\delta)$-crossings made by $X$, is equal to the number of jumps made by $Z$. Thus we need to prove that $Y$ makes at least as many jumps as $Z$.

Recall the one step recursion for $Y$ and $Z$. We claim that in order to prove that $Y$ makes at least as many jumps as $Z$, it suffices to prove that for any values of $x, z,a$ such that $a$ satisfies \eqref{eq:lambda-a2-1}, \eqref{eq:lambda-a2-2} we have:
\begin{equation}\label{eq:lower-bound-z'}
z'(x,z) \leq_{\A} y'(x,z,a)
\end{equation}
We prove inductively that $Z_k \leq_{\A} Y_{k+1}$ for all $k \geq 0$. The base case is clear since $Z_0 = 1 + \frac{\delta}{2}$ and $Y_1 <  1$. The map $y'$ is monotone, meaning that $y_1 \leq_{\A} y_2$ implies $y'(x,y_1,a) \leq_{\A} y'(x,y_2,a)$. The inductive step then follows by applying \eqref{eq:lower-bound-z'} and then monotonicity::
\[
Z_{k+1} = z'(X_{k+1},Z_k) \leq_{\A} y'(X_{k+1},Z_k,a_{2k+3}) \leq_{\A} y'(X_{k+1},Y_{k+1},a_{2k+3}) = Y_{k+2}
\]
Note that the increments of $z'$ and $y'$ are translation invariant, so without loss of generality we can put $x=0$. The results hold for all values of $a_{k}$ satisfying the assumptions, so we suppress the variable $a$ and write $y(w)$ for $y(0,w,a)$.

To prove \eqref{eq:lower-bound-z'}, consider first the case when $y'(z)$ does not make a jump, which by Remark \ref{rm:y-jump} and \eqref{eq:lambda-a2-2} means $z<1$. This in particular means that $z'$ does not make a jump, so $z' = \max\{z, -1-\frac{\delta}{2}\}$. Since $y'$ is increasing, we have $y'(z) > z$. On the other hand, since $z \geq_{\A} -\infty$, we have:
\[
y'(z) \geq y'(-\infty) = \frac{1}{\abs{\log\lambda}}\log\frac{1}{a^2}\lambda \geq -1-\frac{\delta}{8}
\]
Thus, $y' \geq \max\{z,-1-\frac{\delta}{8}\} \geq z'$, finishing the claim.

The other case is when $y'$ makes a jump from $\A_{2i}$ to $\A_{2i+1}$, so that $z \geq 1$. Then either $z < 1 + \frac{\delta}{2}$, so $z'$ remains in $\A_{2i}$ and the claim is trivial, or $z \geq 1 + \frac{\delta}{2} \in \A_{2i}$. In that case $z' = 1+ \frac{\delta}{2} \in \A_{2i+1}$ and we need to prove $y' \geq_{\A} 1 + \frac{\delta}{2} \in \A_{2i+1}$, that is, $y' \leq 1+ \frac{\delta}{2}$ since we are in the odd numbered component. Since $z \geq_{\A} 1 + \frac{\delta}{2}\in \A_{2i}$, by monotonicity of $y'$ we have:
\begin{align*}
y'(z) \leq y'\left(1+\frac{\delta}{2}\right) = \frac{1}{\abs{\log\lambda}}\log\frac{(1 - \frac{\lambda^2}{a^2})\lambda^{-1-\frac{\delta}{2}} + \frac{1}{a^2}\lambda}{\lambda^{-\frac{\delta}{2}} - 1} \leq \\
1+\frac{\delta}{2} + \frac{1}{\abs{\log\lambda}}\log\frac{1 + \frac{1}{a^2}\lambda^{2+\frac{\delta}{2}}}{\lambda^{-\frac{\delta}{2}} - 1} <
1+\frac{\delta}{2} + \frac{1}{\abs{\log\lambda}}\log\frac{1 + \lambda^{2+\frac{3}{8}\delta}}{\lambda^{-\frac{\delta}{2}} - 1}
\end{align*}
It suffices to make the expression under the logarithm smaller than $1$. This is easily implied by $\lambda < 1$ and $\lambda^{-\frac{\delta}{2}} > 3$, which follows from the assumption that $\lambda < \left(\frac{1}{3}\right)^{\frac{2}{\delta}}$.
\end{proof}

We note that for even values of $n$, we obtain similar statements, but with roles of down- and upcrossings reversed. In particular, we obtain an analogue of Proposition \ref{prop:intro-finite} for even $n$, which is identical but bounds the number of jumps in terms of upcrossings, rather than downcrossings of $S_k$.

%% file: sec3b.tex
\subsection{Convergence to the scaling limit}\label{sec:conv-limit}

%so one can think of it as an exponentiated random walk, $\approx i e^{S_k}$, where $S_k$ is a random walk.
%We would like the movement of the center to converge to a continuous path, so we introduce variables rescaled as in the Central Limit Theorem: $S_k = \sqrt{n} X_k$. We would like to show that $z_k$ also converge to a piecewise continuous path and find the limiting trajectory. Note that to this end, we need to rescale $\lambda$ in the same way as the movement of the center, i.e. put $\lambda \approx e^{-\sqrt{n}}$, as otherwise the constant rotation rate $\lambda$ would suppress the fluctuations of the center's movement.

We now study the continuous-time scaling limit of the processes $Y$. To this end, we put $\lambda = e^{-\sqrt{n}}$ and consider a sequence of processes $Y^{(n)}$ such that the underlying functions $X^{(n)}$ converge to some function $X$. The reason for this particular scaling is related to Brownian scaling which we shall use in Section \ref{section:main} when introducing the probabilistic part of the analysis. In the formula below the reader should recognize the same scaling as in the Central Limit Theorem. 

We introduce the continuous time version of the discrete process $X$ with scaling $\lambda=e^{-\sqrt{n}}$. For any $\eta$, we can write:
\begin{align*}
&X^{(n)}_{\eta n}(t) := \frac{1}{\sqrt{\eta^2 n}}S^{(n)}_{\eta^2 n t} 
%&Y_{n}(t) := \frac{1}{\sqrt{n}}\log z_{ n t}
\end{align*}
We introduce the superscript $n$ to emphasize that for different values of $n$, the sums $S^{(n)}_k$ depend on different sets of edge weights $\{a_{k}^{(n)}\}_{k=1}^{n}$ for each $n$. While it may seem natural to consider sequences obtained from a single infinite operator $H$ restricted to finite intervals, we will need the main theorem of this section stated in generality.

We will establish the main theorem of this section under the following assumptions.

\begin{condition}\label{cond:functional-clt}
$X^{(n)}_{n}(t)$ converge uniformly on some interval $[0,T]$ to a function $X(t)$ such that $X(0)=0$.
\end{condition}

\begin{condition}\label{cond:xm2}
%Whenever $M(t)-X(t) = 2$ (resp. $m(t)-X(t) = -2$), for any $\varepsilon > 0$ there exists $t' < t + \varepsilon$ such that for some $\delta > 0$ we have $M(s)-X(s) > 2$ (resp. $m(s)-X(s) < -2$) for $s \in [t',t'+\delta]$.
Fix some $\nu > 0$. Whenever $M(t)-X(t) = 2\nu$ (resp. $m(t)-X(t) = -2\nu$), for any $\varepsilon > 0$ there exists $t' < t + \varepsilon$ such that $M(t')-X(t') > 2\nu$ (resp. $m(t')-X(t') < -2\nu$). 
\end{condition}

\begin{condition}\label{cond:limsup}
We have:
%\[
%\lim_{n\rightarrow \infty}\max_{k=1,\dots,n}\left\{ \frac{\abs{\alpha_k}}{\sqrt{n}}\right\} = 0
%\]
%Equivalently, for the choice of variables as in Section \ref{sec:transfer}, we have:
\begin{equation}\label{eq:an-limsup}
\lim_{n\rightarrow \infty}\max_{k=1,\dots,n} \frac{\abs{\log\vert a^{(n)}_k \vert}}{\sqrt{n}} = 0
\end{equation}
\end{condition}

Note that Condition \ref{cond:limsup} is not implied by Condition \ref{cond:functional-clt} since Condition \ref{cond:functional-clt} only implies that $\abs{\log\abs{\frac{a^{(n)}_{2k-1}}{a^{(n)}_{2k}}}}$ are small in the limit and implies nothing about $\abs{\log\vert a^{(n)}_k \vert}$.

\begin{theorem}\label{th:convergence}
Let $X^{(n)}_{n}(t)$ be a sequence of piecewise linear functions that converge uniformly on some interval $[0,T]$ to a function $X(t)$ such that $X(0)=0$ (Condition \ref{cond:functional-clt}). Suppose that $X(t)$ satisfies Condition \ref{cond:xm2} with $\nu=1$ and Condition \ref{cond:limsup} is satisfied. Let $J_n$ denote the number of jumps made by $\{Y^{(n)}_{k}\}_{k=1}^{n}$ and let $J$ denote the number of $2$-crossings made by $X$. Then $J_n$ converges to $J$ as $n\rightarrow\infty$.
%Then for any $\varepsilon > 0$ $Y_n(t)$ converges to $Y(t)$ uniformly outside sets $(\tau_i-\varepsilon, \tau_i + \varepsilon)$. In particular, $J_n \rightarrow J$ and $\tau^{n}_{i} \rightarrow \tau_i$.
\end{theorem}

\begin{proof}
Let $C^{\alpha}_n$ denote the number of $\alpha$-crossings made by $X_n$, let $C^{\alpha}$ denote the number of $\alpha$-crossings made by $X$ and let $B^{\delta}_{n}$ denote the number of $k$ such that $\abs{\log\abs{a^{(n)}_{k}}} > \frac{\delta}{16} \sqrt{n}$. Pick $\delta >0$. By applying Propositions \ref{prop:crossings} and \ref{prop:discrete-lower-bound} with $\lambda=e^{-\sqrt{n}}$, for large enough $n$ we obtain:
\begin{equation}\label{eq:both-bounds}
C^{2+\delta}_n - 2 B^{\delta}_n \leq J_n \leq C^{2-\delta}_n + 2 B^{\delta}_n
\end{equation}
Since $X_n$ converge to $X$ uniformly, it is easy to see that for large enough $n$ we have $C^{2+\delta}_n \geq C^{2+2\delta}$ and $C^{2-\delta}_n \leq C^{2-2\delta}$. Moreover, Condition \ref{cond:xm2} implies that $C^{2-2\delta}$ and $C^{2+2\delta}$ converge to $C^{2}=J$ as $\delta \rightarrow 0$. Also, by Condition \ref{cond:limsup} we have that $B^{\delta}_{n}$ converges to $0$ as $n \rightarrow \infty$. Thus, we can make the upper and lower bound in \eqref{eq:both-bounds} arbitrarily close to $J$ by first picking small enough $\delta$ and then large enough $n$, depending on $\delta$, which finishes the proof.
\end{proof}

\begin{remark}\label{rm:conv-process}
The proofs of Proposition \ref{prop:crossings} and \ref{prop:discrete-lower-bound} imply that the discrete process $Y$ is always bounded from both sides by $(2\pm\delta)$-crossing processes $Z^{2\pm\delta}$, outside of jump times $\tau_{i}$. This can be actually used to prove the convergence of the processes $Y$ to the $2$-crossing process $Z^{2}$. We do not spell out the details as we only need the convegence of the number of jumps.
\end{remark}

%\begin{remark}\label{remark:jumps-equal-crossings}
%By definition, $\tau_{2i+1}$ are times of subsequent $2$-downcrossings of $X(t)$ and $\tau_{2i}$ are times of subsequent $2$-upcrossings of $X(t)$. In particular, $Y(t)$ makes a jump exactly at a time when $X(t)$ makes a $2$-upcrossing or $2$-downcrossing.
%\end{remark}

%Let $\tau_0 = 0$,  and:
%\begin{align*}
%\tau_{2i+1} &= \inf\{t > \tau_{2i} : M(\tau_{2i},t)-X(t) \geq 2\}\\
% \tau_{2i} &= \inf\{t > \tau_{2i-1} : m(\tau_{2i-1},t)-X(t) \leq -2\}
%\end{align*}

We are now ready to prove the Lemma which summarizes the connection between the asymptotic number of eigenvalues of Schroedinger operators described in Section \ref{sec:transfer} and the processes introduced above.

\begin{definition}\label{def:well}
We say that an interval $[a,b] \subseteq [0,T]$ is an $s$-well of $X$ for some $s > 0$ if: 1. $b=T$ or $X(a)=X(b)$; 2. $X(a) > X(x)$ for all $x \in (a,b)$ and $X(a) - X(x) \geq s$ for some $x \in (a,b)$. 
\end{definition}

\begin{remark}\label{rm:wells-and-crossings}
Note that the maximal number of disjoint $s$-wells made by $X$ inside $[0,T]$ is equal to the number of $s$-downcrossings.
\end{remark}

\begin{figure}[h!]
  \centering
\includegraphics[scale=0.5]{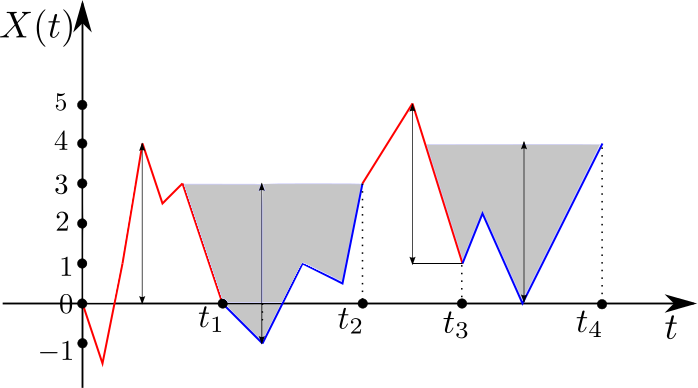}
  \caption{$X(t)$ makes a $4$-downcrossing at $t_1$ and $t_3$ and $4$-upcrossing at $t_2$ and $t_4$. The resulting $4$-wells are shown in grey.}
\end{figure}

%Let $U_i =  2\log\abs{\frac{a_{2i}}{a_{2i-1}}}$.

\begin{lemma}\label{lm:det-local}
Let $\{H^{(n)}\}_{n=1}^{\infty}$ be a family of Schroedinger operators, each given by edge weights $\{a^{(n)}_i\}_{i=1}^{n}$. Let $H^{(n)}_{m}$ denote the operator $H^{(n)}$ restricted to the interval $[1,m+1]$.  For $t,\eta > 0$, let $\Lambda_n(t,\eta)$ denote the number of eigenvalues of $H^{(n)}_{\lfloor tn\rfloor}$ inside the interval $[0,e^{-\eta\sqrt{n}}]$. Consider a subsequence of $n$ such that $\lfloor tn \rfloor$ is odd. Suppose that Conditions \ref{cond:functional-clt}. \ref{cond:xm2} and \ref{cond:limsup} are satisfied, including Condition \ref{cond:xm2} with $\nu=\eta$. In particular, $X^{(n)}_{n}$ converge to some function $X$. Let $\Lambda(t,\eta)$ denote the number of $2\eta$-downcrossings of $X$ inside the interval $[0,\frac{1}{2}t]$. Then $\Lambda_n(t,\eta)$ converges to $\Lambda(t,\eta)$ as $n\rightarrow \infty$.
\end{lemma}

\begin{proof}
Recall that:
\begin{align*}
S^{(n)}_{k} &= \sum^{k}_{i=1}U^{(n)}_{i} = \sum^{k}_{i=1}2\log\abs{\frac{a^{(n)}_{2i-1}}{a^{(n)}_{2i}}}\\
X^{(n)}_n(t) &= \frac{1}{\sqrt{n}}S^{(n)}_{nt}
\end{align*}
where the superscript $n$ emphasizes that, for different $n$, the sums $S^{(n)}$ depend on a different set of edge weights $a_{i}^{(n)}$. 

Put $T=\frac{1}{\eta^2}t$. By letting $n' = \eta^2 n$, $\Lambda_n(t,\eta)$ is equal to the number of eigenvalues of $H^{(n)}_{\lfloor Tn' \rfloor}$ inside the interval $[0,e^{-\sqrt{n'}}]$. By Lemma \ref{lm:topo}, this is equal to $\lceil \frac{1}{2}J_n\rceil$, where $J_n$ is the number of jumps made up to time $\frac{1}{2}Tn'$ by the process $\{Y^{(n),\lambda}\}_{k=1}^{\frac{1}{2}Tn'}$, with $\lambda=e^{-\sqrt{n'}}$ and the underlying random walk being $X^{(n)}_{n'}$. Note that the scaling factor $n'$ is now different than the sample number $n$. The factor of $\frac{1}{2}$ appears since each step of the discrete process used in Section \ref{section:setup} actually corresponds to two steps of the eigenvalue counting process from Section \ref{sec:transfer}(see equation \eqref{eq:rotation}). 

We have:
\[
X^{(n)}_{n'}(t) = \frac{1}{\sqrt{n'}}S^{(n)}_{n't} = \frac{1}{\eta}\cdot\frac{1}{\sqrt{n}}S^{(n)}_{\eta^2nt}
\]
Since by assumption $X^{(n)}_n(t) \rightarrow X(t)$, we have $X^{(n)}_{n'}(t) \rightarrow X_{\eta}(t) := \frac{1}{\eta}\cdot X(\eta^2 t)$.
%Let $Y_{\eta}$ denote the limiting process with the underlying random walk being $X_{\eta}$.

Since $X$ satisfies Condition \ref{cond:xm2} with $\nu=\eta$, it follows that $X_{\eta}$ satisfies the same condition with $\nu=1$. Since the assumptions of Theorem \ref{th:convergence} are satisfied for $X_{\eta}$, the number of jumps made by $\{Y^{(n),\lambda}\}_{k=1}^{\frac{1}{2}Tn'}$ converges to the number of $2$-down- or upcrossings that $X_{\eta}$ makes inside the interval $[0,\frac{1}{2}T]$. This is the same as the number of $2\eta$-crossings made by $X$ inside $[0,\frac{1}{2}t]$. This finishes the proof by noting that if $C$ is the number of $2$-down- or upcrossings, then $\lceil \frac{1}{2}C\rceil$ is the number of $2$-downcrossings since the first crossing is always a downcrossing.
\end{proof}

%% file: sec4.tex
\section{Local statistics and the expected spectral measure at zero}\label{section:main}

In this section, we use the limiting process from the previous section to derive the limit of local eigenvalue statistics of $H_n$ and compute $\mu_H$, the expected spectral measure of the random Schroedinger operator $H$, near zero. The main focus of this section are Theorem \ref{th:local}, Lemma \ref{lm:main} and Theorem \ref{th:main}. We shall use the notation from Section \ref{section:convergence}. Whenever we speak of convergence in distribution of processes, we shall mean weak convergence of measures on $C([0,T])$.

\subsection{Local eigenvalue statistics}\label{sec:main-local}

In this Section, we consider a random Schroedinger operator with edge weights given by a sequence of random variables $\{a_i\}_{i=1}^{\infty}$. We define:
\begin{align*}
&U_i = 2\log\abs{\frac{a_{2i-1}}{a_{2i}}}, \
S_n = \sum_{i=1}^{n}U_i, \
X_n(t) = \frac{1}{\sqrt{n}}S_{\lfloor nt\rfloor}
\end{align*}

\begin{definition}\label{def:clt}
Suppose that $\{U_i\}_{i=1}^{\infty}$ satisfy $\E U_i = 0$ and $\E U^2_i  < \infty$. We say that $U_i$ satisfy the functional Central Limit Theorem if $X_n$ converge in distribution on some interval $[0, T]$ to a Brownian motion $X$ with mean zero.
\end{definition}

\begin{proposition}\label{prop:limsup}
Suppose that $\{a_i\}_{i=0}^{\infty}$ are identically distributed and $\E(\log\abs{a_{i}})^2 < \infty$. Then for any $\varepsilon > 0$ we have:
\[
\Pp\left(\max_{k=1,\dots,n} \frac{\abs{\log\vert a_k \vert}}{\sqrt{n}} > \varepsilon\right) \rightarrow 0
\]
i.e. the sequence of variables $\{\max_{k=1,\dots,n} \frac{\log\vert a_k \vert}{\sqrt{n}}\}_{n=1}^{\infty}$ converges to $0$ in probability.
\end{proposition}

\begin{proof}
Let $X_k = \log^2\abs{a_{i}}$, so that $\E X_k < \infty$. A simple exercise in probability (see e.g. \cite{billingsley-pm}, Problem 21.3) shows that when variables $X_k$ are identically distributed with $\E X_1 < \infty$, we have $\E \max_{k=1,\dots,n}X_k = o(n)$. This together with Markov's inequality proves the desired claim.
\end{proof}

If the functional Central Limit Theorem is satisfied, we can use Lemma \ref{lm:det-local} and Skorokhod almost sure representation to obtain  convergence in distribution of the local eigenvalue statistics. This describes the limiting local eigenvalue distribution around $0$ in terms of crossings made by a Brownian motion, recall Definition \ref{def:crossing}. We note that the limiting distribution is not Poisson, as Proposition \ref{prop:tail} shows that it exhibits a Gaussian tail.

\begin{theorem}\label{th:local}
Suppose that $U_i$ satisfy the functional Central Limit Theorem and let $\Var \ \log\abs{a_{i}} =\sigma^2  < \infty$. Let $H_m$ be the operator $H$ restricted to the interval $[1,m+1]$. For $t,\eta > 0$, let $\{\Lambda_n(t,\eta), t,\eta > 0\}$ denote the random process equal to the number of eigenvalues of $H_{\lfloor tn \rfloor}$ inside the interval $[0,e^{-\eta\sqrt{n}}]$ and let $\{\Lambda(t,\eta), t,\eta > 0\}$ denote the random process equal to the maximal number of disjoint $\eta$-wells made by a Brownian motion with variance $\sigma^2$ inside the interval $[0,t]$. Consider a subsequence of $n$ such that $\lfloor tn \rfloor$ is odd.

Then the process $\Lambda_n(t,\eta)$ converges to the process $\Lambda(t,\eta)$, i.e. all finite dimensional distributions of $\Lambda_n$ converge weakly to finite dimensional distributions of $\Lambda$.
\end{theorem}

%\begin{lemma}\label{lm:m1-weakconv}
%Let $J_{n}$, resp. $J$, be the number of times the process $Y_{n}$, resp. $Y$, makes a jump. If $U_i$ are identically distributed, satisfy the %functional Central Limit Theorem and Condition \ref{cond:moment}, then $J_{n}$ converge in distribution to $J$.
%\end{lemma}

\begin{proof}
By assumption, $X_n$ converge weakly to $X$, which is a Brownian Motion with variance $8\sigma^2$. By Skorokhod representation theorem (\cite{kallenberg}, Theorem 4.30), we can find a common probability space such that $X_n \rightarrow X$ uniformly almost surely. The maximum process $M-X$ is equal in distribution to a reflected Brownian motion and it is standard (\cite{brownian}) that almost surely the Brownian motion has no isolated zeros. These two facts easily imply that $X$ almost surely satisfies Condition \ref{cond:xm2} for any value of $\nu$. Moreover, by Proposition \ref{prop:limsup} and almost sure representation of variables converging in probability, we obtain that Condition \ref{cond:limsup} is satisfied almost surely. Thus, almost surely the conditions of Theorem \ref{th:convergence} are satisfied.

By Lemma \ref{lm:det-local}, for any $t,\eta$ it holds that $\Lambda_n(t,\eta)$ converges to the number of $2\eta$-wells that $X$ makes inside the interval $[0,\frac{1}{2}t]$. By Brownian scaling this has the same distribution as the number of $\eta$-wells that a Brownian motion with variance $\sigma^2$ makes inside the interval $[0,t]$, which is exactly $\Lambda(t,\eta)$. The same statement holds for $\Lambda_n$ treated as a process - for any finite number of $\{(t_1,\eta_1),\dots,(t_k,\eta_k)\}$ convergence holds simultaneously almost surely for all $t_i,\eta_i$, which implies weak convergence of all finite dimensional distributions.

Note that the limiting $2$-crossing process $Z^2$ is discontinuous, but is right continuous with left limits. Then Remark \ref{rm:conv-process} also implies that the processes $Y^{(n)}$ converge weakly to $Z^2$. We do not describe this in detail as below we only need the weak convergence of the number of eigenvalues and not of the processes $Y^{(n)}$ themselves.
\end{proof}

\begin{remark}\label{rm:smallest-eigen}
Various questions about the eigenvalues of $H_n$ can be phrased in terms of crossings made by the limiting Brownian Motion $X$. For example, let $\lambda^{(n)}_0$ be the smallest positive eigenvalue of $H_n$. It follows from the preceding Theorem that $\frac{-\log\lambda^{(n)}_0}{\sqrt{n}}$ converges in distribution to the maximal size of a downcrossing made by $X$ inside the interval $[0,1]$, that is, the maximal $\eta$ such that $M(t)-X(t) = \eta$ for some $t\in[0,1]$. Since $M(t)-X(t)$ is equal in distribution to a reflected Brownian motion $\abs{X(t)}$, we obtain that:
\[
\frac{-\log\lambda^{(n)}_0}{\sqrt{n}} \ \Longrightarrow \ \sigma \cdot \sup_{t \in [0,1]}\abs{B(t)}
\]
in distribution, where $B(t)$ is the standard Brownian Motion of variance $1$.
\end{remark}

\begin{proposition}\label{prop:tail}
With $\Lambda(t,\eta)$ defined as in Theorem \ref{th:local}, for some constants $a, b > 0$ depending on $t, \eta,\sigma$ we have
\[
\Pp(\Lambda(t,\eta) > m ) \leq a e^{-bm^2}
\]
\end{proposition}

\begin{proof}
Suppose that $\Lambda(t,\eta) > m$, so $X(t)$ has made more than $2m$ $\eta$-up- or down-crossings inside the interval $[0,t]$. By Brownian scaling, we can put $\eta'= \frac{\eta \sqrt{t}}{\sigma}$ and instead consider $\eta'$-crossings of a standard Brownian motion inside the interval $[0,1]$ Let $T_1, \dots, T_{2m}$ be the times between subsequent $\eta$-up- or down-crossings. The times $T_i$ are independent and identically distributed. Moreover, note that $R(t) = M(t)-X(t)$ has the distribution of a reflected Brownian motion, so by the reflection principle we have:
\begin{align*}
\Pp(T_i \leq x) = \Pp(\max_{[0,x]} R(t) \geq \eta') \leq 2 \cdot \Pp(\max_{[0,x]} B(t) \geq \eta')
= 4 \cdot \Pp(B(x) \geq \eta') \leq 4 e^{-\frac{\eta'^2}{2 x}}
\end{align*}

Since there are more than $2m$ crossings inside $[0,1]$, this means that $T_{1} + \dots T_{2m} < 1$. Let $\Delta_{2m} = \{(x_1,\dots,x_{2m}) \in [0,1]^{2m} \vert \sum_{i=1}^{2m}x_{i} \leq 1\}$. We have:
\begin{align*}
&\Pp(T_{1} + \dots T_{2m} < 1) \leq
\int_{\Delta_{2m}}\prod_{i=1}^{2m}\Pp(T_{i} \in dx_{i}) \leq
4\int_{\Delta_{2m}}\prod_{i=1}^{2m} e^{-\frac{\eta'^2}{2x_{i}}} =
4\int_{\Delta_{2m}}e^{-\frac{\eta'^2}{2}\sum_{i=1}^{2m}\frac{1}{x_{i}}}
\leq \\
&4\int_{\Delta_{2m}}e^{-2m^2\eta'^2\frac{1}{\sum_{i=1}^{2m}x_{i}}} \leq
4 \cdot \mathrm{vol}(\Delta_{2m})\cdot e^{-2\eta'^2\cdot m^2} = 4 \cdot \frac{e^{-2\eta'^2\cdot m^2}}{(2m)!}
\end{align*}

\end{proof}

\subsection{The expected spectral measure at zero}\label{sec:main-exp}

Having established weak convergence of the local eigenvalue statistics, we now set the background to compute the expected spectral measure of $H$. We shall consider sequences of $U_i$ satisfying the following technical condition:
\begin{condition}\label{cond:moment}
At least one of the following holds:
\begin{enumerate}
\item for some $\gamma > 2$ we have $\E\abs{\log\abs{a_{i}}}^{\gamma} < \infty$ and for some constant $C$ and any $n \geq 1$:
\begin{equation}\label{eq:moment-condition}
\E\abs{U_1 + \dots + U_n}^{\gamma} \leq C n^{\frac{\gamma}{2}},
\end{equation}
\item The $a_i$ are i.i.d. and $\E\abs{\log\abs{a_{i}}}^{2} < \infty$.
\end{enumerate}
\end{condition}
Note that in the i.i.d. case a standard computation verifies the inequality \eqref{eq:moment-condition} with $\gamma = 2$. However, we will later work in the setting when $U_i$ are dependent, but exhibit a suitable correlation decay (see Section \ref{sec:hyperbolic}).

Having established weak convergence of the number of jumps, we would like to establish that the convergence also holds in expectation. This will be implied by the upcoming Lemma \ref{lm:unif-integrable}. An important element in the proof is the maximal inequality of \cite{moricz}, whose special case says the following:
\begin{proposition}\label{prop:moricz}
Let $S(i,j) = \sum_{k=i}^{j}U_k, M(i,j) = \max_{i\leq k \leq j} S(i,k)$. If for some $\gamma > 2$ and $A > 0$ we have:
\[
\E \abs{S(i, j)}^{\gamma} \leq A (j-i+1)^{\frac{\gamma}{2}}
\]
for all $j > i \geq 1$, then:
\[
\E \abs{M(i,j)}^{\gamma} \leq c_{\gamma } \cdot A (j-i+1)^{\frac{\gamma}{2}}
\]
for all $j > i \geq 1$, with the constant $c_{\gamma}$ depending only on $\gamma$.
\end{proposition}

\begin{lemma}\label{lm:unif-integrable}
Suppose that $U_i$ satisfy Condition \ref{cond:moment}. Let $J_n$ be the random variable equal to the number of jumps made by $Y^{(n)}$ in the interval $[1,n]$.  Then the family of variables $\{J_n\}_{n\geq 1}$ is uniformly integrable.
\end{lemma}

\begin{proof}
By Proposition \ref{prop:crossings} with $\delta = 1$ and $\lambda = e^{-\sqrt{n}}$, we obtain $J_n \leq 2 B_n + C_n$ for all $n$ greater than some global constant $n_0$, where:
\[
B_n = \sum_{i=1}^{n}\id_{\{ \abs{\log\abs{a_{i}}}> \frac{1}{16}\sqrt{n}\}}
\]
is the number of times $i$ such that $\abs{\log\abs{a_{i}}} > \frac{1}{16}\sqrt{n}$ and $C_n$ is the number of $1$-crossings made by $X_n$. Thus, it suffices to show that $B_n$ and $C_n$ are uniformly integrable. We will first prove the general case. Then we give a different argument for the i.i.d. case under weaker moment assumptions.

%To show uniformly integrability we need to prove that for any $\varepsilon > 0$ there exists an $L > 0$ such that $\E J_n \id_{\{J_{n}>L\}} < \varepsilon$. Obviously $J_{n}\leq n$, so we only need to consider $n > L$. We can write:
%\[
%\E J_n \id_{\{J_{n}>L\}} = \sum_{k=L+1}^{\infty}\Pp(J_n \geq k) = \sum_{k=L+1}^{n}\Pp(J_n \geq k)
%\]
%We now apply Proposition \ref{prop:crossings} with $\delta = 1$ and $\lambda = e^{-\sqrt{n}}$. We obtain $J_n \leq 2 B_n + C_n$ for all $n$ greater than some global constant $n_0$, where:
%\[
%B_n = \sum_{i=1}^{n}\id_{\{ \abs{\log\abs{a_{i}}}> \frac{1}{16}\sqrt{n}\}}
%\]
%is the number of times $i$ such that $\abs{\log\abs{a_{i}}} > \frac{1}{16}\sqrt{n}$ and $C_n$ is the number of $1$-crossings made by $X_n$. Thus by union bound:
%\[
%\sum_{k=L+1}^{n}\Pp(J_n \geq k) \leq \sum_{k=L+1}^{n}\Pp(2 B_n  + C_n \geq k) \leq \sum_{k=L+1}^{n}\Pp\left(B_n \geq \frac{k}{4}\right)+\sum_{k=L+1}^{n}\Pp\left(C_n \geq \frac{k}{2}\right)
%\]

To prove that $B_n$ is uniformly integrable, it suffices to show that $\E B_n \rightarrow 0$. By Condition \ref{cond:moment} we can estimate:
\begin{align*}
 \E B_n = n \cdot \Pp\left( \abs{\log\abs{a_{i}}}> \frac{1}{16}\sqrt{n}\right) \leq
 n \cdot \frac{\E   \abs{\log\abs{a_{i}}}^{\gamma} }{   n^{\frac{\gamma}{2}} (\frac{1}{16})^{\gamma}  }.
\end{align*}
Since $\gamma > 2$, this converges to $0$ as $n \rightarrow \infty$.

%Obviously it suffices to prove that $\Pp(J_n \geq k) < s_k$ with $s_k$ independent of $n$ and $\sum_{k=1}^{\infty} s_k < \infty$. We will first prove such a statement in the general case for $s_k \approx k^{-1-\varepsilon}$. Then we give a different argument for the i.i.d. case, giving $s_k \approx c^k$ for some $c < 1$.

%Let $S(i,j) = \sum_{k=i}^{j}U_k, M(i,j) = \max_{i\leq k \leq j} S(i,k), m(i,j) = \min_{i\leq k \leq j} S(i,k)$. We shall say that the process $S_i$ has made a double crossing in an interval $I=[i,j] \subseteq [1,n]$, if there exists a $ k \in I$ such that:
%\begin{align*}
%& k = \inf \{l > i :  M(i, l) - S_l = \sqrt{n}\} \\
%& j = \inf \{l > k : m(k ,l) - S_l = -\sqrt{n} \}
%\end{align*}

For a crossing $[s,t]$, its length is $t-s$. To show that $C_n$ are uniformly integrable, it suffices to show that the expected number of $1$-crossings $[s,t]$ of length at most $n\cdot2^{-\ell_0}$ is at most $\varepsilon(\ell_0)$ for all $n$, where $\varepsilon(\ell_0)\to 0$ as $\ell_0\to\infty$. Indeed, for given $\ell_0$ the number of $1$-crossings that are longer is bounded independently of $n$.

We first bound the expected number $e(n,\ell)$ of $1$-crossings that are of length at least $n\cdot2^{-\ell}$ and at most $2n\cdot2^{-\ell}$. To each such crossing $[s,t]$ we associate the first time of the form $k \cdot \lfloor 2^{-\ell}n \rfloor$ contained in $[s,t]$ for some integer $k$. Each such time is contained in at most one such crossing, so the number of such crossings is bounded above by the sum of the indicators over such times that a crossing contains that time.

As a consequence, $e(n, \ell)$ is bounded by the number of such times, which is $2^\ell$, times the maximal probability that a crossing contains a given time. If this happens, then the range (i.e. the maximum minus the minimum) of the process $S_k$ over the interval $[i,j]$ of size $4n \cdot 2^{-\ell}$ centered at that time is at least $\sqrt{n}$. In particular, either $M(i,j)\ge \sqrt{n}/2$, or the same holds for the absolute minimum. By Markov's inequality, Proposition \ref{prop:moricz} and Condition \ref{cond:moment} we get
$$
\Pp(M(i,j)\ge \sqrt{n}/2) \le
 \frac{ 2^{\gamma/2} \E|M(i,j)|^\gamma }{n^{\gamma/2}}
 \le
 \frac{ c \cdot\E|S(i,j)|^\gamma }{n^{\gamma/2}}
 \le
 \frac{c'\cdot (4n\cdot 2^{-\ell} )^{\gamma/2}}{n^{\gamma/2}}
$$
giving $e(n,\ell)\le 4^{\gamma/2}c'\cdot2^{\ell} 2^{-\ell \gamma/2}$. As a result, the expected number of $1$-crossings of length at most $n\cdot2^{-\ell_0}$ is at most
$$
c\sum_{\ell=\ell_0}^{\lceil \log_2 n\rceil}
2^{\ell(1-\gamma/2)}
<
\frac{c \cdot 2^{\ell_0(1-\gamma/2)}}{1-2^{(1-\gamma/2)}}
=:
\varepsilon(\ell_0) \to 0
$$
as long as $\gamma>2$.
\end{proof}
\begin{proof}[Proof of the i.i.d. case]
We now prove the i.i.d. case assuming only $\E\abs{\log\abs{a_{i}}}^{2} < \infty$. It suffices to show that the variance of $B_n$ is uniformly bounded in $n$. The $B_n$ are Bernoulli random variables, so their variance is bounded by their mean. We have
\begin{align*}
\E(B_n) = n\cdot\Pp\left(\abs{\log\abs{a_{1}}}> \frac{1}{16}\sqrt{n}\right)\le
n \cdot \frac{\E \abs{\log\abs{a_{i}}}^{2} }{n (\frac{1}{16})^2},
\end{align*}
which is bounded.

To show that $C_n$ is uniformly integrable, it suffices to show exponential tails uniform in $n$, namely a bound of the form $\Pp(C_n \geq k) < c^k$ for some $c < 1$. Let $a$ be a large constant to be chosen later. We divide the interval $[1,n]$ into $a$ intervals $I_1, \dots, I_a$ of size $\frac{n}{a}$. Let $N_{n,j}$ denote the number of crossings with starting points in $I_j$. It suffices to show that each $N_{n,j}$ has exponential tails uniform in $n$. 

If the walk makes at least $2$ crossings inside $I_j=[s,t]$, then the maximal absolute increment between times $s,t$ in this interval is at least $\sqrt{n}$. With the notation of Proposition \ref{prop:moricz} either $M(s,t)\ge \sqrt{n}/2$ or the same holds for the absolute minimum.
By the classical Kolmogorov maximal inequality for i.i.d variables
\[
\Pp(M(s, t) > \lambda) \leq \frac{(s-t) \cdot \Var \ U_i}{\lambda^2},
\]
which gives
\begin{equation}\label{eq:smaller-than-c}
P(N_{i,j} \ge 2) \le 2 \cdot\frac{n/a\cdot\Var \ U_i}{n/4}.
\end{equation}
By choosing $a$ large enough, we can make the right hand side smaller than some global constant $b < 1$, which we fix from now on.

We can write for any $i$:
\begin{align*}
\Pp(N_{n,j} \geq 2i) = \prod_{k=1}^{i}\Pp(N_{n,j} \geq 2k \ \vert \ N_{n,j} \geq 2k-2)
\end{align*}
If we stop the random walk in the time interval $I_j$ after the $2k-2$nd crossing (when it exists),  the conditional law of the remaining stretch is another independent random walk in a smaller interval. The conditional probability of the remaining walk making a double crossing is at most
$b$ again. This shows that all terms in the above product are bounded above by $b$, so $N_{n,j}$ has exponential tails uniform in $n$, as required. 
\end{proof}

The crucial step in computing the spectral measure of $H$ is approximation by an operator $H_{n}$ on a finite interval. The spectral measure of $H_n$ is computed in the following lemma.

\begin{lemma}\label{lm:main}
Fix $K > 0$ and let $n = \lfloor K^2 \abs{\log \varepsilon}^2 \rfloor$ be odd (in particular, $n\rightarrow \infty$ as $\varepsilon\rightarrow 0$). Let $H_n$ be equal to the operator $H$ restricted to the interval $\{1, \dots, n+1\}$ and let $\mu_{n}$ be its expected spectral measure. Assume that for all $i$ $\Var \ \log\abs{a_{i}} =\sigma^2  < \infty$. Suppose that:
\begin{enumerate}
\item $U_i$ satisfy the functional Central Limit Theorem (Definition \ref{def:clt})
\item $U_i$ satisfy Condition \ref{cond:moment}
\end{enumerate}
Then there exists $c_K$ such that $\lim_{K\rightarrow\infty}c_K = 0$ and for sufficiently small $\varepsilon$:
\[
\abs{\mu_{n}(-\varepsilon, \varepsilon) - \frac{\sigma^2}{\abs{\log^2 \varepsilon}}} \leq \frac{c_K}{\abs{\log^2 \varepsilon}}
\]
\end{lemma}

\begin{proof}
%We add subscript $K$ whenever we want to emphasize the dependence on $K$. We shall count the eigenvalues of $H_{n}$ using the limiting process $Y(t)$ described in the previous sections. Note the factor of $2$ in $n = 2K^2 \abs{\log \varepsilon}^2$ is required for consistency with the scaling $\lambda = e^{-\frac{1}{K}\sqrt{n'}}$ used in Section \ref{section:setup}, for $n' = \frac{1}{2}n$. This is due to the fact that each step of the discrete process used in Section \ref{section:setup} actually corresponds to two steps of the eigenvalue counting process from Section \ref{sec:transfer}(see equation \eqref{eq:rotation}).

%By Lemma \ref{lm:m1-weakconv}, $Y_n$ converge in distribution to $Y$.

Let $M_{K,n}$ denote the number of eigenvalues of $H_{n}$ inside the interval $[-\varepsilon, \varepsilon]$ and let $M^{+}_{K,n}$ denote the number of eigenvalues inside the interval $[0, \varepsilon]$. We start by ruling out zero eigenvalue of $H_n$. The equation for eigenvectors (see Section \ref{sec:transfer}) with $\lambda=0$ can be solved recursively and for odd $n$ the only solution is identically zero. Since we assumed $n$ is odd, all the eigenvalues of $H_n$ are thus nonzero. Since the underlying graph is bipartite, every eigenvalue $\lambda_{+} \in (0, \varepsilon]$ has a corresponding eigenvalue $\lambda_{-} =-\lambda_{+} \in [-\varepsilon, 0)$, so:
\[
M_{K,n} =2 M^{+}_{K,n}
\]

Let $J_{K}$ denote the number of $\frac{1}{K}$-crossings that a Brownian motion with variance $\sigma^2$ makes inside the interval $[0,1]$. By Brownian scaling this is the same as the number of $1$-crossings that a Brownian motion with variance $\sigma^2$ makes inside the interval $[0,K^2]$. By assumption, $U_i = 2\log\abs{\frac{a_{2i-1}}{a_{2i}}}$ satisfy the functional Central Limit Theorem. Since $\varepsilon = e^{-\frac{1}{K}\sqrt{n}}$, by putting $t=1, s=\frac{1}{K}$ in Theorem \ref{th:local} we obtain that $M^{+}_{K,n}$ converges weakly to $\left\lceil\frac{1}{2}J_{K}\right\rceil$. Thus we obtain:
\[
M_{K,n} = 2 M^{+}_{K,n} \Rightarrow 2\left\lceil\frac{1}{2}J_{K}\right\rceil = J_{K} +  \id_{\{J_{K}=2k+1\}}
\]
By Lemma \ref{lm:unif-integrable} the family $\{M_{K,n}\}_{n \geq 1}$ is uniformly integrable since $M_{K,n}$ differs from the number of jumps by at most one. Weak convergence thus implies convergence in expectation, so by letting:
%\[
%\E M_{K,n} \rightarrow \E J_{K} + \E  \id_{\{J_{K}=2k+1\}}
%\]
\begin{equation}\label{eq:lemma-akn}
a_{K,n} = \E M_{K,n} - \E J_{K} - \E  \id_{\{J_{K}=2k+1\}}
\end{equation}
for each fixed $K$ we have $\lim_{n \rightarrow \infty}a_{K,n} = 0$.

We shall now compute $\E J_{K}$. We shall use standard properties of Brownian motion, which can be found e.g. in \cite{brownian}.  We first compute $\tau$, the expected time to make an up or down crossing. Let $X$ be a Brownian motion with variance $\sigma^2$. Consider $\tau_1 = \inf\{t\geq 0 : M(0,t) - X(t) = 1\}$. Since $M(0,t)-X(t)$ has the distribution of a reflected Brownian motion, $\tau_{1}$ has the same distribution as $\tau_2 = \inf\{t \geq 0 : X(t) = 1 \vee X(t) = -1\}$. It is standard that $\E \tau_{2} < \infty$. If $X(t)$ is stopped at time $\tau_2$, by Wald identity we obtain $\E X(\tau_2)^2 = \sigma^2 \cdot \E \tau_2$, which gives $\E \tau_2 = \frac{1}{\sigma^2}$. Clearly we can treat the event $\{m(0,t) - X(t) = -1\}$ in the same way, so we have:
\begin{equation}\label{eq:lemma-tau}
\tau = \E \tau_{2} = \frac{1}{\sigma^2}
\end{equation}

Now, let $T_{i}$ be i.i.d. random variables equal to the times between successive up or down crossings made by $X(t)$. Note that $\E T_{i} = \tau$ and we can write $J_{K}$ as:
\[
J_{K} = \sup\{n : \sum_{i=1}^{n}T_{i} \leq K^2\}
\]
In other words, $J_{K}$ is equal to the number of jumps up to time $K^2$, made by a renewal process with expected renewal time $\tau$. By the law of large numbers for renewal processes \cite{durrett}, it follows that $\lim_{K\rightarrow \infty}\frac{\E J_{K}}{K^{2}} = \frac{1}{\tau}$. Let:
\begin{equation}\label{eq:lemma-bk}
b_K = \frac{\E J_{K}}{K^{2}} - \frac{1}{\tau}
\end{equation}
so that $\lim_{K\rightarrow\infty}b_K = 0$.

In this way we obtain:
\begin{align*}
&\mu_{n}(-\varepsilon, \varepsilon) =
\frac{\E M_{K, n}}{n}  \stackrel{\eqref{eq:lemma-akn}}{=}
\frac{\E J_{K} + \E  \id_{\{J_{K}=2k+1\}} + a_{K,n}}{n} =\\
&\frac{\frac{\E J_{K}}{K^2}}{\abs{\log^2 \varepsilon}} +
\frac{\E  \id_{\{J_{K}=2k+1\}} + a_{K,n}}{n} \stackrel{\eqref{eq:lemma-bk}}{=}
\frac{\frac{1}{\tau} + b_{K}}{\abs{\log^2 \varepsilon}} +
\frac{\E  \id_{\{J_{K}=2k+1\}} + a_{K,n}}{n} \stackrel{\eqref{eq:lemma-tau}}{=} \\
%&\frac{\sigma^2}{\abs{\log^2 \varepsilon}} + \frac{b_K}{\abs{\log^2 \varepsilon}} +  \frac{\E  \id_{J_{K}=2k+1} + a_{K,n}}{n} = \\
& \frac{\sigma^2}{\abs{\log^2 \varepsilon}} + \frac{1}{\abs{\log^2 \varepsilon}}(b_K + \frac{1}{K^2}(\E  \id_{J_{K}=2k+1} + a_{K,n}))
\end{align*}
Recall that $\lim_{n\rightarrow \infty}a_{K,n} = 0$, and $n = K^2\abs{\log\varepsilon}^2$, so for fixed $K$ for sufficiently small $\varepsilon$ $a_{K,n} < 1$. Also, $\E  \id_{\{J_{K}=2k+1\}} \leq 1$. Thus, we can let $c_{K} = b_{K} + \frac{3}{K^2}$ and the lemma is proved.
\end{proof}

We are now ready to prove our main theorem.

\begin{theorem}\label{th:main}
Let $\mu_H$ be the expected spectral measure of $H$ as defined by Definition \ref{def:spectral-measure}. With notation and assumptions of Lemma \ref{lm:main} the following holds:
\[
\mu_H(-\varepsilon, \varepsilon) = \frac{\sigma^2}{\abs{\log^2 \varepsilon}}(1 + o_{\varepsilon}(1))
\]
%where $\sigma^2 = \Var U_i$.
\end{theorem}

\begin{proof}
Let $n = K^2\abs{\log \varepsilon}^2$ be equal to $n=2^k$ for some $k$, with $K > 0$ to be chosen later.

By Proposition \ref{prop:spectral-exists} and choice of $n=2^k$, we have $d_K(\mu_H,\mu_n) \leq \frac{1}{n}$. By the definition of the Kolmogorov distance, this implies:
%\[
%\abs{\mu_{n}(-\varepsilon, \varepsilon) - \mu_{T}(-\varepsilon, \varepsilon)} \leq
%\]
%where $\mathrm{rank}(A) = 1 - \mu_{A}(\{0\})$. Now, $A = T-T'$ has nonzero edge weights only at $0, \pm n, \pm 2n, \dots$. Therefore, in every finite restriction $A_{kn}$ of $A$ to an interval of size $kn$ the eigenvalue $0$ will appear with proportion $\geq \frac{kn-k}{kn}=1-\frac{1}{n}$, so by convergence of spectral measure at $0$ we have $\mu_{A}(\{0\}) \geq 1 - \frac{1}{n}$. In this way we arrive at:
\[
\abs{\mu_{n}(-\varepsilon, \varepsilon) - \mu_{H}(-\varepsilon, \varepsilon)} \leq \frac{2}{n} = \frac{2}{K^2\abs{\log^2 \varepsilon}}\\
\]
By Lemma $\ref{lm:main}$, there exists some $\varepsilon(K)$ such that for all $\varepsilon < \varepsilon(K)$ we have:
\[
\abs{\mu_{n}(-\varepsilon, \varepsilon) -  \frac{\sigma^2}{\abs{\log \varepsilon}^2}} \leq \frac{c_K}{\abs{\log^2 \varepsilon}}
\]
Using the triangle inequality and multiplying by $|\log \varepsilon|^2/\sigma^2$ we get
\begin{align*}
&\abs{\mu_{H}(-\varepsilon, \varepsilon)\cdot\frac{\abs{\log \varepsilon}^2}{\sigma^2} -  1} \leq \frac{1}{\sigma^2}\left(\frac{2}{K^2} + c_K\right).
\end{align*}
Since $\lim_{K\rightarrow \infty}c_K = 0$, for every $\delta > 0$ we can find $K(\delta)$ such that the right hand side is smaller than $\delta$ for all $\varepsilon < \varepsilon(K(\delta))$, which proves the claim.
\end{proof}

We now discuss how the above result relates to Dyson's results from \cite{dyson}. Dyson considers random variables $\lambda_j$ and the matrix $\Lambda$ defined by $\Lambda_{j+1,j} = -\Lambda_{j,j+1} = i \lambda_{j}^{\frac{1}{2}}$ and zero otherwise. He then proceeds to compute the function $M(z)$, defined as the fraction of eigenvalues of $\Lambda$ inside the interval $(-\sqrt{z},\sqrt{z})$. To translate to our setting, we put $a_i = \lambda_i^{\frac{1}{2}}$ in the definition of the operator $H_n$ and note that $\Lambda$ is conjugate to $H_n$ by a diagonal matrix $A$ with $A_{k,k} = (-i)^{k-1}$. We then get $M(z) = \mu_{H}(-\varepsilon,\varepsilon)$ for $\varepsilon = \sqrt{z}$.

In Section VI, Dyson computes the asymptotics of $M(z)$ explicitly for $\lambda_j$ drawn from the probability distribution $G_{n}(\lambda) = \frac{n^n}{(n-1)!}\lambda^{n-1}e^{-n\lambda}$, where $n\geq 1$ is an integer parameter. These asymptotics can be easily recovered from Theorem \ref{th:main} by simply computing the variance $\Var \ \log\abs{a_{i}}$ with $a_i = \sqrt{\lambda_i}$. For example, for $n=1$ $\lambda_j$ are exponential random variables and by computing the variance of their logarithm we obtain $\mu(-\varepsilon,\varepsilon) \sim \frac{C}{\abs{\log\varepsilon}^2}$ for $C = \frac{1}{4}\cdot\frac{\pi^2}{6}$, which is in agreement with Dyson's explicit computation (formula (72) in \cite{dyson}; the factor $\frac{1}{4}$ comes from $\varepsilon=\sqrt{z}$ in the formula for $M(z)$).

%% file: sec5.tex
\section{Spectral measures for groups}\label{sec:groups}

\subsection{Random Schroedinger operators from group ring elements}\label{sec:graboluk}

In this section, we use random Schroedinger operators to study spectral measures of group ring elements. The construction and exposition below is based on \cite{graboluk}. We assume familiarity with Pontryagin duality for Abelian groups.

Let $\Gamma$ be a discrete group, $M$ a discrete Abelian group and $\rho: \Gamma \curvearrowright M$ an action of $\Gamma$ on $M$. Let $X=\widehat{M}$ denote the Pontryagin dual of $M$. Each $m \in M$ determines a function $\widehat{m}: X \rightarrow \C$, given by $m(x)=x(m)$, and by linearity we can extend this to $\C[M]$, i.e. to any $f \in \C[M]$ we associate $\widehat{f}:X \rightarrow \C$, which we shall call the Fourier transform of $f$.

Given $\rho$, we also have a dual action $\widehat{\rho}: \Gamma \curvearrowright X$, given by $(\widehat{\rho}(\gamma)(x))(m)=x(\rho(\gamma^{-1})\cdot m)$. If we choose a generating set $S$ for $\Gamma$, we can consider the Schreier graph $\schr(\Gamma,X,S)$ associated to the dual action $\widehat{\rho}$.

We consider the semidirect product corresponding to the action $\rho$, i.e. the group $G = M \rtimes \Gamma$. Let $H \in \C[G]$ be a self-adjoint group ring element, which we identify with the corresponding self-adjoint operator $H: \ell^2(G) \rightarrow \ell^2(G)$. Every such element can be written as:
\[
H = \sum_{\gamma}\gamma \cdot f_{\gamma}
\]
where $f_{\gamma} \in \C[M]$. Let $\widehat{f_{\gamma}}: X \rightarrow \C$ denote the dual of $f_{\gamma}$. 

For $x \in X$, let $\schr(x)$ denote the connected component of the Schreier graph $\schr(\Gamma,X,S)$ containing $x$. We define $H_x: \ell^2(\schr(x))\rightarrow \ell^2(\schr(x))$ as the convolution operator on $\schr(x)$ with the following edge labels. The label of edge from $y$ to $\widehat{\rho}(\gamma)\cdot y$ is given by $\widehat{f_{\gamma}}(y)$. Note that since $X$ is equipped with the Haar measure $\nu$, we can treat $H_x$ as a random operator, where $x$ is chosen randomly according to $\nu$.

For a self-adjoint $H: \ell^2(G)\rightarrow \ell^2(G)$, $\mu_{H}$ will denote its spectral measure, i.e. the unique measure such that for any $k \geq 0$, we have:
\[
\scalar{\delta_e}{H^k \delta_e} = \int_{\R}x^k d\mu_{H}(x)
\]
where $\delta_e \in\ell^2(G)$ equals $1$ at $e$ (the identity element) and zero otherwise.

In order to compute $\mu_H$, we invoke the following theorem \cite{graboluk}, which gives the correspondence between the spectral measure of $H$ and the expected spectral measure of $H_{x}$. 

\begin{theorem}\label{th:graboluk}
The spectral measure $\mu_H$ is equal to the expected spectral measure of the family $H_x$, i.e. for any Borel subset $A$:
\[
\mu_{H}(A) = \int_{X}\mu_{H_{x}}(A)d\nu(x)
\]
\end{theorem}

In the examples we study, we shall take $\Gamma = \Z$ and $f_{\gamma}\neq 0$ only for $\gamma = a, a^{-1}$, where $a=1$ is the standard generator of $\Z$. In that case, the corresponding operator $H_x$ can be easily described. If $x \in X$ is chosen from the Haar measure $\nu$ on $X$, for almost every $x$ $\schr(x)$ will be isomorphic to $\Z$, with the edge weight from $n$ to $n+1$ given by $\widehat{f_{\gamma}}(\widehat{\rho}(a^n)(x))$. Therefore $H_x$ is a random Schroedinger operator on $\Z$, where the randomness in edge weights $\widehat{f_{\gamma}}(\widehat{\rho}(a^n)(x))$ comes from the random choice of $x$.

The above correspondence shows that in order to compute $\mu_H$, it suffices to analyze the expected spectral measure of the random Schroedinger operator $H_x$. We shall now perform this computation for specific examples, using results derived in previous sections.

\subsection{Semidirect products by hyperbolic matrices}\label{sec:hyperbolic}

In this section we consider groups $G$ of the following form. Let $\Z$ act on $\Z \times \Z$ by a hyperbolic matrix $A \in SL(2,\Z)$. For concreteness we take $A=\begin{pmatrix}
 2 & 1 \\
 1 & 1
\end{pmatrix}
$.

Let $G = \Z \times \Z\rtimes_{A} \Z$ be the corresponding semidirect product. Let $s, t$ be the standard generators of $\Z \times \Z$ and let $a$ denote the generator of $\Z$. Note that groups of this type correspond precisely to lattices in the Sol group \cite{sol}.

We consider the switch-walk operator $H \in \C[G]$ shifted by $5 e$:
\begin{align*}
H &= a \cdot (s + t + s^{-1} + t^{-1} + 5e) + (s + t + s^{-1} + t^{-1} + 5e)\cdot a^{-1} 
\end{align*}

We shall prove the following theorem:

\begin{theorem}\label{th:hyperbolic}
The spectral measure of $H \in \C[G]$ satisfies:
\[
\mu_H(-\varepsilon,\varepsilon) = \frac{C}{\abs{\log^2 \varepsilon}}(1 + o_{\varepsilon}(1))
\]
for some constant $C > 0$. In particular, the Novikov-Shubin invariant of $H$ is equal to $0$.
\end{theorem}

We start with a technical lemma needed to ensure that assumptions of the main theorem are satisfied.

\begin{lemma}\label{lm:decorrelation}
Let $\{U_i\}_{i=1}^{\infty}$ be a stationary sequence of random variables such that $\E U_i = 0$, $\E U_i^4 < \infty$. Let $U_{i_{1},\dots,i_{k}} = U_{i_{1}}\cdot\dots\cdot U_{i_{k}}$. Assume that $U_{i_{1},\dots,i_{k}}$ satisfy uniformly exponential correlation decay, i.e. there exist $c > 0$ and $\lambda < 1$ such that for any $l \geq 2$ and any $i_{1}\leq\dots\leq i_{k} \leq i_{k+1} \leq \dots \leq i_{l}$ we have:
\[
\abs{\E U_{i_{1},\dots,i_{k},i_{k+1},\dots,i_{l}} - \E U_{i_{1},\dots,i_{k}}\cdot \E U_{i_{k+1},\dots,i_{l}} } \leq c \lambda^{i_{k+1}-i_{k}}
\]
Then $U_i$ satisfy Condition \ref{cond:moment} with $\gamma=4$, i.e. for some constant $C > 0$:
\[
\E(U_1 + \dots + U_n)^4 \leq C n^2
\]
\end{lemma}

\begin{proof}
Let us expand:
\[
\E(U_1 + \dots + U_n)^4 = \sum_{1 \leq i_{1},\dots,i_{4} \leq n}\E U_{i_{1}, \dots, i_{4}} 
\]
Let us order the indices so that $i_1 \leq \dots \leq i_4$. We define a block of size $l$ to be a maximal set of indices $I = \{i_{k}, i_{k+1}, \dots, i_{k+l-1} \}$ such that $\abs{i_{k+j}-i_{k+j-1}} \leq  b\log n$ for $j=1,\dots,l-1$, with $b$ chosen so that $c\lambda^{b\log n}=n^{-2}$. 

Consider a term in which there is a block containing only a single index $i_k$. The number of such terms is at most $n^4$. By correlation decay, for each such term we can write:
\[
\E U_{i_{1},\dots,i_{k-1},i_{k},\dots,i_{4}} \leq \E U_{i_{1},\dots,i_{k-1}}\cdot \E U_{i_{k},\dots,i_{4}} + n^{-2} 
\leq \E U_{i_{1},\dots,i_{k-1}}\cdot (\E U_{i_{k}}\cdot\E U_{i_{k+1},\dots,i_{4}} + n^{-2}) + n^{-2}
\]
Since $\E U_{i} = 0$, the contribution from a single such term is $O(n^{-2})$, so the total contribution from such terms is $O(n^2)$. 

If there is no block with only a single index, then there is either one block $\{i_1,i_2,i_3,i_4\}$ or two blocks $\{i_1,i_2\},\{i_3,i_4\}$. The number of one block terms is at most $n \cdot (b\log n)^3 = O(n^2)$, so we only need to bound the contribution from the two block terms. Note that for any $k$:
\[
\E U_{i_{k}i_{k+1}} \leq c\lambda^{i_{k+1} - i_{k}}
\]
Since $i_3 - i_2 > b \log n$, we have:
\[
\E U_{i_1,i_2,i_3,i_4} \leq \E U_{i_1,i_2}\cdot \E U_{i_3,i_4} + n^{-2} \leq c^2 \cdot\lambda^{i_2-i_1} \cdot \lambda^{i_4 - i_3} + n^{-2} 
\]
The second term is again bounded by the number of possible terms $n^4$ times $n^{-2}$, which gives $O(n^2)$. For the first term on the right hand side, there are $n$ choices for $i_1$ and $i_3$, and once these are chosen, summation over possible values of $i_2$ and $i_4$ gives a geometric series bounded by $\frac{1}{1-\lambda}$, so we obtain that the total contribution is also $O(n^2)$. 
\end{proof}

\begin{proof}[Proof of Theorem \ref{th:hyperbolic}]
We apply the construction from Section \ref{sec:graboluk}. The Pontriagin dual of $\Z \times \Z$ is equal to $\mathbb{T} = S^1 \times S^1$, with the dual action $\widehat{\rho}$ given by $\widehat{\rho}(a)(x) = A^Tx$. To simplify expressions we define $B=A^T$. Let $f = s+t+s^{-1}+t^{-1} + 5e$. The dual $\widehat{f}: \mathbb{T} \rightarrow \C$ is given by:
\[
\widehat{f}(z_1, z_2) = 2\rre{z_1} + 2\rre{z_2} + 5
\]
By construction from Section \ref{sec:graboluk}, the random Schroedinger operator corresponding to $H$ is given as follows. The random edge weight from $n$ to $n+1$ is given by $a_{n} = \widehat{f}(B^n x)$, where $x \in \mathbb{T}$ chosen uniformly from the Haar measure.

It suffices to check that the random Schroedinger operator defined as above satisfies the assumptions of Theorem \ref{th:main}. Writing  $U_i = \log\abs{\frac{a_{2i-1}}{a_{2i}}}$ explicitly, we obtain:
\[
U_i = \log\left\vert\frac{\widehat{f}(B^{2i-1} x)}{\widehat{f}(B^{2i} x)}\right\vert
\]
Letting $\phi: \mathbb{T}\rightarrow \R$ be equal to:
\begin{equation}\label{eq:phi}
\phi(y) = \log\left\vert\frac{2 \rre (B^{-1} y_1) + 2 \rre (B^{-1} y_2) + 5}{2 \rre (y_1) + 2\rre (y_2) + 5}\right\vert
\end{equation}
and $S=B^2$ we can write $U_i = \phi(S^{i} x)$. Note that, thanks to the shift by $5e$, $\phi$ is a well defined $C^{\infty}$-function on $\mathbb{T}$.

We first check that $U_i$ satisfy the functional Central Limit Theorem.  We say that $\phi: \mathbb{T} \rightarrow \R$ is a coboundary if there exists a measurable $h$ such that $\phi = h - h \circ S$. We can invoke the functional Central Limit Theorem proved for actions of toral automorphisms in \cite{hyperbolic-clt}, which can be stated as follows:
\begin{proposition}\label{prop:borgne}
Let $S: \mathbb{T} \rightarrow \mathbb{T}$ be a toral map generated by a hyperbolic matrix. Let $\phi \in L^2(\mathbb{T})$ with Fourier series:
\[
\phi(\cdot) = \sum_{k \in \Z^2}c_k \cdot e^{2i\pi \scalar{k}{\cdot}}
\]
such that $c_0 = \int_{\mathbb{T}} \phi = 0$.

If $\phi$ is not a coboundary and the Fourier coefficients $c_k$ satisfy for some $A> 0$ and $\theta > 2$:
\begin{equation}\label{eq:borgne}
\abs{c_{(k_1,k_2)}} \leq A\prod_{i=1}^{2}\frac{1}{(1 + \abs{k_i})^{\frac{1}{2}} \log^{\theta}(2+ \abs{k_i}) }
\end{equation}
then the functional Central Limit Theorem holds for the sequence $\{S^i\phi\}_{i=0}^{\infty}$.
\end{proposition}

This statement is implied by the main theorem of \cite{hyperbolic-clt} via Remark 1 therein.

Since $\phi$ as defined is \eqref{eq:phi} is a smooth function, its Fourier coefficients $c_{k}$ decay faster than any polynomial $\abs{k}^{-\alpha}$, so in particular condition \eqref{eq:borgne} is satisfied. 

It remains to check that $\phi$ is not a coboundary. To this end, it suffices to exhibit a periodic orbit $\{x, Sx,\dots,S^{k-1}x\}$ such that $\sum_{i=0}^{k-1} \phi(S^i x) \neq 0$. Recalling that $S=B^2$ and $\phi(y) = g(B^{-1}y)-g(y)  $ for $g(y) = \log\abs{2 \rre y_1 + 2\rre y_2 + 5}$, this is equivalent to:
\begin{equation}\label{eq:coboundaries}
g(x)+g(B^2x) + \dots +g(B^{2k}x) \neq g(Bx) + g(B^3 x) + \dots + g(B^{2k+1}x)
\end{equation}
Recall that $B=A^T = \begin{pmatrix}
 2 & 1 \\
 1 & 1
\end{pmatrix}$. It is readily checked that the set $\{(\frac{1}{3},0),(\frac{2}{3},\frac{1}{3}),(\frac{2}{3},0),(\frac{1}{3},\frac{2}{3})\}$ is periodic and corresponds to the set of points $\{(e^{\frac{4}{3}i\pi},1),(e^{\frac{4}{3}i\pi},e^{\frac{3}{3}i\pi}),(e^{\frac{4}{3}i\pi},1),(e^{\frac{2}{3}i\pi},e^{\frac{2}{3}i\pi})\} \subseteq \mathbb{T}$. After computing both sides of \eqref{eq:coboundaries} we conclude that they are not equal, which proves that $\phi$ is not a coboundary. Note that this is the only step of the proof where we use the specific form of the matrix $A$ and the same proof will hold for different choices of $A$, provided one can prove that $\phi$ is not a coboundary (e.g. by exhibiting a suitable periodic orbit).

The second step is to verify that $U_i$ satisfy Condition \ref{cond:moment}. This is taken care of by Lemma \ref{lm:decorrelation}, provided we can establish uniform exponential correlation decay. Actions of hyperbolic matrices are well known to satisfy such exponential decay of correlations for smooth observables. We use the main theorem from \cite{ruelle}, which implies the following as a special case:

\begin{proposition}\label{prop:ruelle}
Let $S: \mathbb{T} \rightarrow \mathbb{T}$ be a toral map generated by a hyperbolic matrix. Then there exist $C,k >0$ such that if $\phi',\phi'':\mathbb{T}\rightarrow \R$ are $C^1$ functions, we have:
\[
\abs{\mu((\phi' \circ S^{-m'})\cdot(\phi'' \circ S^{-m''})) -\mu(\phi') \cdot \mu(\phi'')} \leq C \Vert \phi' \Vert_{C^{1}} \cdot \Vert \phi'' \Vert_{C^{1}} \cdot e^{-k\abs{m'-m''}} 
\]
where $\mu(f)=\int_{\mathbb{T}}fd\mu$ and $\mu$ is the Lebesgue measure.
\end{proposition}

We use the notation of Lemma \ref{lm:decorrelation}. Let $l \geq 2$ and let $i_{1}\leq \dots\leq i_{k} \leq i_{k+1}\leq\dots \leq i_{l}$ be such that $\abs{i_{k+1} - i_{k}} \geq n$. Define:
\begin{align*}
& \phi'(x) = U_{i_{1}+n,\dots,i_{k}+n} = \prod_{j=1}^{k}U_{i_{j}+n} = \prod_{j=1}^{k}\phi(S^{i_{j}+n} x)\\
& \phi''(x) = U_{i_{k+1},\dots,i_{l}} = \prod_{j=k+1}^{l} U_{i_{j}} = \prod_{j=k+1}^{l}\phi(S^{i_{j}} x)
\end{align*}
so that $\E U_{i_{1},\dots,i_{k}} = \mu(\phi' \circ S^{-n}) = \mu(\phi')$ and $\E U_{i_{k+1},\dots,i_{l}} = \mu(\phi'')$.

With this notation, we have:
\[
U_{i_{1},\dots,i_{l}} = \prod_{j=1}^{l} U_{i_{j}} = (\phi' \circ S^{-n}) \cdot \phi''
\]
and $\E U_{i_{1},\dots,i_{l}} = \mu((\phi' \circ S^{-n}) \cdot \phi'')$. Note that $\Vert \phi' \Vert_{C^{1}}, \Vert \phi'' \Vert_{C^{1}}$ are bounded by a global constant that depends only on $l$ and $\Vert \phi \Vert_{C^{1}}$. Applying Proposition \ref{prop:ruelle} for $m'=0, m''=n$ proves that the assumptions of Lemma \ref{lm:decorrelation} are satisfied, so $U_i$ satisfy Condition \ref{cond:moment}.
\end{proof}

\subsection{Lamplighter groups}\label{sec:lamplighter}

We shall now describe how a similar approach can be used for computing spectral measures of lamplighter groups. We start with the standard lamplighter group $\Z_2 \wr \Z = \oplus_{\Z}\Z_2 \rtimes \Z$.

Let $\Z_2 = \{e, t\}$. Let $e_0$ denote the element of $\oplus_{\Z}\Z_2$ that has $e$ at every position and let $t_i$ denote the element of $\oplus_{\Z}\Z_2$ that has $t$ at position zero and $e$ elsewhere. Consider the switch-walk operator $H \in \C[\Z_2 \wr \Z]$ given by:
\[
H = a \cdot (p\cdot e_0 + (1-p)\cdot t_0) + (p\cdot e_0 + (1-p)\cdot t_0)\cdot a^{-1} 
\]
so letting $f_a = p\cdot e_0 + (1-p)\cdot t_0$ we have:
\[
H = a \cdot f_a + f_a \cdot a^{-1}
\]
This operator corresponds to the random walk on $G$ where at each step we move either left or right and then either leave the current lamp intact with probability $p$ or change it with probability $1-p$. 

The Pontriagin dual of $M = \oplus_{\Z}\Z_2$ is equal to $X = \prod_{\Z}\Z_2$, with the Haar measure on $X$ being the usual product measure. Since $\rho: \Z \curvearrowright M$ is action by translations, the dual action $\widehat{\rho}: \Z \curvearrowright X$ is given by $(\widehat{\rho}(a)(y))_j = y_{j-1}$, i.e. the Bernoulli shift. Therefore, the edge weights in the corresponding random Schroedinger operator will be i.i.d. with distribution determined by $\widehat{f_a}: \prod_{\Z}\Z_2 \rightarrow \C$. Since $\widehat{e_0} = \id, \widehat{t_0}=(-1)^{y_0}$, for $y \in \prod_{\Z}\Z_2$ we have:
\[
\widehat{f_a}(y) = p + (1-p)(-1)^{y_{0}}
\]
so the edge weight is equal to $1$ or $2p-1$ with probability $\frac{1}{2}$ each.

We consider $p \neq \frac{1}{2}$, as otherwise the relevant random variables $\log\abs{\frac{a_{2i-1}}{a_{2i}}}$ are infinite with positive probability (note that $p=\frac{1}{2}$ gives edge weights $0$ or $1$, i.e. the edge percolation on $\Z$). For such an operator we can apply Theorem \ref{th:main} and obtain the spectral measure at zero to be $\mu_H(-\varepsilon,\varepsilon) \approx \frac{C}{\abs{\log\varepsilon}^2}$, with $C = \frac{1}{4}(\log\abs{2p-1})^2$.

A similar approach can be used for general lamplighter groups of the form $G = \Lambda \wr \Z$, where the lamp group $\Lambda$ is not necessarily Abelian or finite. Let $\Lambda$ be generated by a set $S$ closed under inverses and let $\lambda = \frac{1}{\abs{S}}\sum_{s \in S}s_0$. Consider the switch-walk operator $H \in \C[G]$ given by:
\begin{align*}
H &= a \cdot \lambda + \lambda\cdot a^{-1} 
\end{align*}
Such an operator corresponds to the random walk on $G$ where at each step we move either left or right and then change the lamp by performing a step of a simple random walk in $\Lambda$.

In this setting, it is known (Grabowski and Vir{\'a}g, unpublished) that the random Schroedinger operator corresponding to $H$ is obtained by putting i.i.d. weights on the edges, each drawn from the distribution given by $\mu_{\lambda}$, the spectral measure of the simple random walk on $\Lambda$. As before, we are in position to use Theorem \ref{th:main} and get $\mu_H(-\varepsilon,\varepsilon) \approx \frac{C}{\abs{\log\varepsilon}^2}$.

Note that this method can be extended to $H$ of more general form, where instead of $\lambda = \frac{1}{\abs{S}}\sum_{s \in S}s_{0}$ we consider general elements $f \in \C[\oplus_{\Z} \Lambda]$, in particular, involving $s_i$ for $i\neq 0$. In that case, the edge weights in the corresponding random Schroedinger operator will not be independent anymore. However, what remains true is the edge weights are obtained by a factor of i.i.d. process. More precisely, if $f$ contains only $s_i$ for $\abs{i} \leq k$, the edge weights in the operator will be obtained as follows - we put i.i.d. labels on the edges and then each edge is assigned a weight that depends only on the labels in a neighborhood of that edge of size $k$. Thus, the obtained weights are weakly dependent and the same technique as in Section \ref{sec:hyperbolic} can be used to obtain the theorem in this case.